\documentclass[11pt,a4paper]{article}
\usepackage{a4wide}
\usepackage[english]{babel}
\usepackage{amsmath,amssymb}
\usepackage{bm}
\usepackage{graphicx}
\usepackage{paralist,multirow,colortbl}
\usepackage[numbers]{natbib}
\usepackage{setspace}
\usepackage[charter]{mathdesign}

\newcommand{\bres}{\tilde{b}^{\textit{res}}}
\newcommand{\E}{\mathbb{E}}
\renewcommand{\P}{\mathbb{P}}
\newcommand{\btilde}{\tilde{b}}
\newcommand{\Btilde}{\widetilde{B}}
\newcommand{\ii}{i=1,\dots,N}
\newcommand{\jj}{j=1,\dots,N}

\newcommand{\equaldist}{\,{\buildrel d \over =}\,}
\newcommand{\limdist}{\,{\buildrel d \over \rightarrow}\,}
\newcommand{\V}{\textrm{Var}}

\newcommand{\ul}{\underline}

\newtheorem{theorem}{Theorem}[section]
\newtheorem{property}[theorem]{Property}
\newtheorem{lemma}[theorem]{Lemma}
\newtheorem{definition}[theorem]{Definition}
\newtheorem{remark}[theorem]{Remark}

\newtheorem{corollary}[theorem]{Corollary}
\newtheorem{conjecture}[theorem]{Conjecture}
\newenvironment{proof}{\par\noindent\textbf{\upshape Proof:}}{\hfill$\square$\par}
\numberwithin{equation}{section}

\graphicspath{{./},{images/}}

\providecommand{\href}[2]{#2}

\title{Heavy traffic analysis of roving server networks}
\author{M.A.A. Boon\footnote{Eurandom and Department of Mathematics and Computer Science, Eindhoven University of Technology, P.O. Box 513, 5600MB Eindhoven, The Netherlands. Email: m.a.a.boon@tue.nl.} \and R.D. van der Mei\ \footnote{Department of Mathematics, Section Stochastics, VU University, De Boelelaan 1081a, 1081HV Amsterdam, The Netherlands and Centre for Mathematics and Computer Science (CWI), 1098 SJ Amsterdam, The Netherlands. Email: \href{mailto:mei@cwi.nl}{mei@cwi.nl}.} \and E.M.M. Winands \footnote{University of Amsterdam, Korteweg-de Vries Institute for Mathematics, Science Park 904, 1098 XH  Amsterdam, The Netherlands. Email: \href{mailto:e.m.m.winands@uva.nl}{e.m.m.winands@uva.nl}.}}

\begin{document}
\maketitle

\begin{abstract}
This paper studies the heavy-traffic (HT) behaviour of queueing networks with a single roving server. External customers arrive at the queues according to independent renewal processes and after completing service, a customer either leaves the system or is routed to another queue. This type of customer routing in queueing networks arises very naturally in many application areas (in production systems, computer- and communication networks, maintenance, etc.). In these networks, the single most important characteristic of the system performance is oftentimes the path time, i.e. the total time spent in the system by an arbitrary customer traversing a specific path. The current paper presents the first HT asymptotic for the path-time distribution in queueing networks with a roving server under general renewal arrivals. In particular, we provide a strong conjecture for the system's behaviour under HT extending the conjecture of Coffman et al. \cite{coffman95,coffman98} to the roving server setting of the current paper. By combining this result with novel light-traffic asymptotics we derive an approximation of the mean path-time for arbitrary values of the load and renewal arrivals. This approximation is not only highly accurate for a wide range of parameter settings, but is also exact in various limiting cases.


\end{abstract}

\section{Introduction}

This paper considers heavy-traffic (HT) limits for queueing networks with  a single roving server that visits the queues in a cyclic order according to the gated and exhaustive service. Customers from the outside arrive at the queues according to general renewal processes, and the service time and switch-over time distributions are general as well. After receiving service at queue $i$, a customer is either routed to queue $j$ with probability $p_{i,j}$, or leaves the system with probability $p_{i,0}$. This model can be seen as an extension of the classical polling model (in which customers always leave the system upon completion of their service) by customer routing.

The vast majority of the polling literature assumes that the arrival process at each queue follows a Poisson process (see, for example, \cite{boonapplications2011,Grillo1,levy1,takagi3} for overviews of polling systems and their applications). In many applications, however, the interarrival times are not exponentially distributed. Therefore, in this paper we study a network in which the arrival process at each of the queues is a general renewal process. For open networks an important characteristic of the system performance is the path time, defined as the total time spent in the system by an arbitrary customer traversing a specific path. That is, oftentimes service level agreements are made on the total time required for all the service requests of a customer in a system to be finished. Moreover, in many computer-communication and production-inventory systems the single most important performance measure is often not an aggregate measure like the mean path time, rather the probability that this path time exceeds a pre-defined threshold. In view of dimensioning such systems, the importance of the path time distribution as a performance measure of interest is evident. Due to the routing of customers, which leads to non-renewal arrival processes at the queues and to strong interdependence of the waiting times within a path time, simulation appears to be the only practical recourse at the present time. In these circumstances, one naturally resorts to asymptotic estimates. In particular, in this paper we study the path-time distribution in a queueing network with customer routing under HT conditions.

The motivation for studying the HT regime - which is also the most challenging regime from a scheduling point of view - is twofold. First, an attractive feature of HT asymptotics is that in many cases they lead to strikingly simple expressions for the performance measures of interest. This remarkable simplicity of the HT asymptotics leads to structural insights into the dependence of the performance measures on the system parameters and gives fundamental insights in the behaviour of the system in general. A second appealing feature of HT asymptotics is that they form an excellent basis for developing simple accurate approximations for the performance measures 
for stable systems.

The introduced queueing network is very general, which is illustrated by the fact that many special cases have been studied in the past. Some special case configurations are standard polling systems \cite{takagi3}, tandem queues \cite{nair,taube}, multi-stage queueing models with parallel queues \cite{katayama}, feedback vacation queues \cite{boxmayechiali97, takine}, symmetric feedback polling systems \cite{takagifeedback,takine}, systems with a waiting room \cite{alineuts84,takacsfeedback77}, and many others. Due to the intrinsic complexity of the model, previous studies on the network in its full generality were restricted to queue lengths and waiting time distributions for stable systems under the assumption of Poisson arrivals (see \cite{boonvdmeiwinandsRovingQuesta,sarkar,sidi1,sidi2}).  Although the results in these papers are exact they lack an explicit analysis with simple  expressions, leading to the need of using numerical techniques to determine performance measures of interest. Moreover, these results are limited to the waiting-time distribution instead of the practically most interesting and theoretically most challenging path-time distribution.

Besides that we have a theoretical interest in the proposed queueing network, the present work is motivated by the fact that customer routing in polling systems arises naturally in a host of application areas (in production systems, computer- and communication networks, maintenance, etc.). Some examples are a manufacturing system where products undergo service in a number of stages or in the context of rework \cite{Grasman1}, a Ferry based Wireless Local Area Network (FWLAN) in which nodes can communicate with each other or with the outer world via a message ferry \cite{kavitha}, a dynamic order picking system where the order picker drops off the picked items at the depot where sorting of the items is performed \cite{gongdekoster08}, and an internal mail delivery system where a clerk continuously makes rounds within the offices to pick up, sort and deliver mail \cite{sarkar}.

Motivated by the attractiveness of HT asymptotics, several approaches have been proposed to obtain HT limits for polling systems. HT limits have been rigorously proven for systems with Poisson arrivals (cf. \cite{olsenvdmei03, RvdM_QUESTA}) and renewal arrivals (cf. \cite{coffman95,coffman98, jennings2010,vdmeiwinands08}. A central role in these papers is the Heavy-Traffic Averaging Principle (HTAP) which means that the total scaled workload may be considered as a constant during a cycle, whereas the workload of the individual queues change much faster according to deterministic trajectories, or a fluid model.  In this paper, we also follow this well established path of deriving HT limits for systems with general renewal arrivals, see also \cite{Markowitz2001,Markowitz2000,reimanwein98,reimanwein99,reimanrubiowein99}.

One of the main contributions of the current paper is that we present the first HT asymptotic for the path-time distribution in queueing networks with a roving server under general renewal arrivals. The main building blocks of this path-time distribution are inevitably the waiting-time distributions at the individual queues, for which the current paper also presents the first exact HT asymptotics\footnote{Some preliminary results restricted to mean waiting times were derived in \cite{boonvdmeiwinandsRovingPER2011}.}. In particular, we provide a strong conjecture for the systems behaviour under HT extending the polling conjecture of Coffman et al. \cite{coffman95,coffman98} to the roving server setting of the current paper. Our conjecture is validated in three ways. Firstly, as stated before we follow a well established and accepted line of thinking. Secondly, our conjecture corresponds to the aforementioned rigorously proven distributional limits in special cases of our network. Thirdly, we give some numerical examples that illustrate that the correct limiting behaviour has been derived. As an important by-product of our analytical framework, we obtain exact  asymptotics in the large switch-over time regime as well.

The second main contribution concerns the derivation of a simple approximation of the mean path-time for arbitrary values of the load and renewal arrivals by combining the HT results with newly derived light-traffic (LT) asymptotics. This approximation technique is shown to be exact in various limiting cases, and is known to be highly accurate for a wide range of parameter settings (\cite{boonapprox2011,dorsman2011}).  Moreover, it satisfies the Pseudo Conservation Law (PCL), and consequently it leads to exact closed-form results of the mean waiting time for symmetric systems with Poisson arrivals. The resulting  expressions are very insightful, simple to implement, and suitable for optimization purposes. In particular, the approximation shows explicitly how the path times depend on the system parameters such as the routing probabilities $p_{i,j}$.

Our presentation of the analysis, and the analysis itself, may be considered to be somewhat informal throughout. Providing a rigorous presentation of our results, however, would be an extremely interesting, but notoriously difficult area for further research. Furthermore, it would take us far afield from our main goal, i.e. to obtain fundamental insights into customer routing in systems. Lastly, we would like to stress that the current paper concerns a continuous-time cyclic system with gated or exhaustive service in each queue, but that all results can be extended to discrete time, to periodic polling, to batch arrivals, or to systems with different branching-type service disciplines such as globally gated service.

The structure of the present paper is as follows. In Section \ref{sect:model} we introduce the model and the required notation. In Section \ref{sect:htap} we use the HTAP to derive exact queue-length, waiting-time, and path-time asymptotics for networks with gated and/or exhaustive service. Based on these results we develop novel approximations in Section \ref{sect:approx}. In the penultimate section, we present some practical cases that illustrate the versatility of the queueing network  and the importance of the path-time distribution in practice. Finally, we present some conclusions and points for discussion in the last section.

\section{Model and notation}\label{sect:model}

In this paper we consider a queueing network consisting of $N\geq2$ infinite buffer queues $Q_1,\dots,Q_N$. External customers arrive at $Q_i$ according to a general renewal arrival process with rate $\lambda_i$, and have a generally distributed service requirement $B_i$ at $Q_i$, with mean value $b_i := \E[B_i]$, and second moment $b_i^{(2)} := \E[B_i^2]$. Throughout this paper we assume that all random variables have finite second moments. The queues are served by a single server in cyclic order. Whenever the server switches from $Q_i$ to $Q_{i+1}$, a switch-over time $R_i$ is incurred, with mean $r_i$. The cycle time $C_i$ is the time between successive moments when the server arrives at $Q_i$. The total switch-over time in a cycle is denoted by $R=\sum_{i=1}^N R_i$ 
and its first two moments are $r:=\E[R]$ and $r^{(2)}:=\E[R^2]$. Indices throughout the paper are modulo $N$, so $Q_{N+1}$ actually refers to $Q_1$.
All service times and switch-over times are mutually independent. Each queue receives exhaustive or gated service. Exhaustive service means that each queue is served until no customers are present anymore, whereas gated service means that only those customers present at the server's arrival at $Q_i$ will be served before the server switches to the next queue. This queueing network can be modelled as a \emph{polling system} with the specific feature that it allows for routing of the customers: upon completion of service at $Q_i$, a customer is either routed to $Q_j$ with probability $p_{i,j}$, or leaves the system with probability $p_{i,0}$. Note that $\sum_{j=0}^N~p_{i,j}=1$ for all $i$, and that the transition of a customer from $Q_i$ to $Q_j$ takes no time. The model under consideration has a branching structure, which is discussed in more detail by Resing \cite{resing93}. The total arrival rate at $Q_i$ is denoted by $\gamma_i$, which is the unique solution of the following set of linear equations: $\gamma_i = \lambda_i + \sum_{j=1}^N \gamma_j p_{j,i}, \ii.$
The offered load to $Q_i$ is $\rho_i:=\gamma_i b_i$ and the total utilisation is $\rho:=\sum_{i=1}^N \rho_i$. We assume that the system is stable, which means that $\rho$ should be less than one (see \cite{sidi2}). The total service time of a customer is the total amount of service given during the presence
of the customer in the network, denoted by $\Btilde_i$, and its first two moments by
$\btilde_i := \E[\Btilde_i]$ and $\btilde_i^{(2)} := \E[\Btilde_i^2]$.
The first two moments are uniquely determined by the following set of linear equations: For $\ii$,
\begin{eqnarray}\label{eq02}
\btilde_i &=& b_i + \sum_{j=1}^N \btilde_j p_{i,j},\\
\btilde_i^{(2)} &=& b_i^{(2)} + 2b_i\sum_{j=1}^N \btilde_j p_{i,j}+ \sum_{j=1}^N \btilde_j^{(2)} p_{i,j}.
\end{eqnarray}

We study this model under heavy-traffic conditions, i.e., we increase the load of the system until it reaches the point of saturation, $\rho\uparrow1$. As the total load of the system increases, the visit times, cycle times, and waiting times become larger and will eventually grow to infinity. For this reason, we scale them appropriately and consider the scaled versions. We consider several variables as a function of the load $\rho$ in the system. For each variable $x$ that is a function of $\rho$, we denote its value {\it evaluated at} $\rho=1$ by $\hat{x}$. Scaling is done by varying the interarrival times of the external customers. To be precise, the limit is taken such that the external arrival rates $\lambda_1,\ldots,\lambda_N$ are increased, while keeping the service and switch-over time distributions, the routing probabilities and the {\it ratios} between these arrival rates fixed. For $\rho = 1$, the generic interarrival time of the stream in $Q_i$ is denoted by $\hat{A}_i$. Reducing the load $\rho$ is done by scaling the interarrival times, i.e., taking the random variable $A_i := \hat{A}_i/\rho$ as generic interarrival time at $Q_i$. Hence, the rate of the arrival stream at $Q_i$ satisfies $\lambda_i = 1/\E[A_i]$.
Furthermore, we define arrival rates $\hat\lambda_i = 1/\E[\hat{A}_i]$, and proportional load at $Q_i$, $\hat\rho_i = {\rho_i}/{\rho}$ (``proportional'' because $\sum_{i=1}^N \hat\rho_i = 1$).

\section{HT asymptotics}\label{sect:htap}

To obtain HT-results for the waiting-time distributions, we use HT results for polling systems without customer routing, which are obtained by \cite{coffman95,coffman98,jennings2010}. The key observation in these papers is the occurrence of a so-called Heavy Traffic Averaging Principle (HTAP). When a polling system becomes saturated, two limiting processes take place. Let $V$ denote the total workload of the system, i.e., the total service requirement of all customers present in the system including the possible residual service time of a customer being served. As the load offered to the system, $\rho$, tends to~1, the scaled total workload $(1-\rho)V$ tends to a Bessel-type diffusion. However, the work \emph{in each queue} is emptied and refilled at a faster rate than the rate at which the total workload is changing. This implies that during the course of a cycle, the total workload can be considered as constant, while the workloads of the individual queues fluctuate according to a fluid model. The HTAP relates these two limiting processes.

Although rigorous proofs have only been presented for standard polling models (see \cite{coffman95,coffman98,jennings2010}), results in \cite{Markowitz2001,Markowitz2000,reimanwein98,reimanwein99,reimanrubiowein99} support the conjecture that the HTAP holds for a much wider class of systems.  As in these papers, we make the crucial assumption that the HTAP holds without providing a rigorous proof of convergence. That is, the HTAP occurs due to a time scale decomposition that is inherent in the heavy traffic scaling.  Therefore, in HT the multi-dimensional individual workload processes move along a path in the constant workload hyperplane. For systems without routing and switch-over times, this path is described in detail by Jennings \cite{jennings2010}. Using the results from this section, this can easily be adapted to our model with customer routing. Furthermore, it is known that the scaled total workload tends to a Bessel-type limit as the system becomes saturated. More colloquially, the routing of customers impacts the individual workloads, which impels us to considerably modify and extend the HT analysis of polling models, but does not affect the time scale decomposition which directly implies that the HTAP should also hold in the current setting.

We provide justification for this conjecture in four ways. Firstly, we follow a well-founded line of thinking. Secondly, our conjecture corresponds to the rigorously proven distributional limit in the special case of a two-queue polling model and branching-type polling models with Poisson arrivals. Thirdly, in the next section we present numerical examples supporting the conjecture that the correct limiting behaviour has been derived. Finally, in the appendix, we provide a theorem and proof for the limiting queue-length distributions under the assumption of Poisson arrivals. In the next subsection we start by discussing the fluid model and subsequently discuss the limiting distribution of the scaled total workload. We use these results to obtain the HT limit of the scaled waiting time and path-time distributions.

\subsection{Fluid model: Gated service}

\paragraph{Workload.}

In this section we consider the fluid version of the queueing network with a single shared server, where the work travels as fluid from one station to another. This fluid model is carefully constructed in such a way that its behaviour corresponds to the multi-dimensional individual workload processes, moving along a path in the constant workload hyperplane as discussed in the introduction of this section. An important aspect of this ``corresponding fluid model'' is the absence of switch-over times, as they become negligible under HT conditions. We start by studying the fluid limit of the per-queue workload, which is obtained by multiplying by $(1-\rho)$ and letting $\rho \uparrow 1$. For our model, the fluid limit of the workload at $Q_i$ is a piecewise linear function. During the visit time of $Q_k$, denoted by $V_k$, $k=1,\dots,N$, \emph{external} fluid particles, corresponding to the customers in the original model, flow into $Q_i$ at rate $\hat\lambda_i$. Each of these fluid particles brings along $\btilde_i$ units of work into the system. Simultaneously, work is being processed in $Q_k$ at rate one. Since $\sum_{i=1}^N\hat\lambda_i\btilde_i=1$, the \emph{total} workload remains constant throughout the course of a cycle. In steady state, the length of one cycle is fixed and denoted by $c$. In this paragraph we will show that there is a simple linear relation between $c$ and the total workload in the system, denoted by $v$. For this reason, we may regard $c$ as a system parameter in the fluid model, instead of $v$.

\begin{figure}[ht]
\begin{center}
\includegraphics[width=0.7\linewidth]{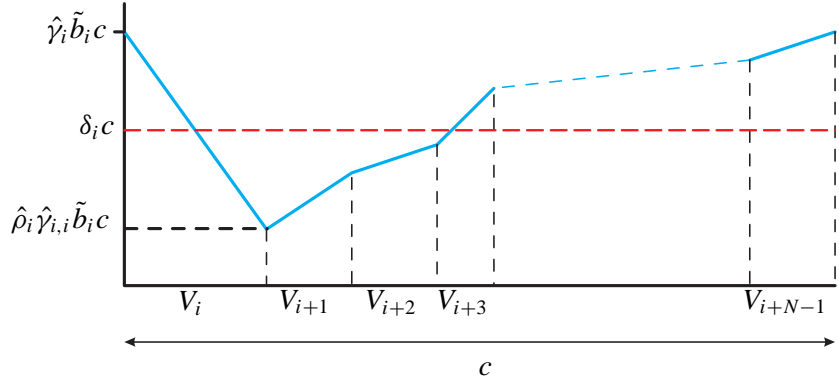}
\end{center}
\caption{Mean amount of work in $Q_i$ in the fluid limit that arises when the system is in heavy traffic. The length of one cycle is $c$.}
\label{fluidHT}
\end{figure}

Although work is processed at rate one, due to the internal routing work is flowing out of $Q_k$ at rate
\[
1+\frac{1}{b_k}\sum_{i=1}^N p_{k,i}\btilde_i = \frac{\btilde_k}{b_k},
\]
which is greater than (or equal to) one. The reason for this anomaly is that work decreases in $Q_k$ either because of the service of fluid particles (customers) in this queue, or because work is shifted due to internal routing of fluid. Work \emph{including} rerouted fluid particles is flowing into $Q_i$, during $V_k$, at rate $\hat\gamma_{i,k}\btilde_i$, where
$\hat\gamma_{i,k}:=\hat\lambda_i+p_{k,i}/{b_k}$, for $i,k=1,\dots,N$.
It is straightforward to verify that ${\btilde_k}/{b_k} = \sum_{i=1}^N \hat\gamma_{i,k}\btilde_i$. Figure \ref{fluidHT} depicts a graphical representation of the mean amount of work in $Q_i$ in the fluid limit throughout the course of a cycle, the length of which is a constant, denoted by $c$.
Using the fact that fluid particles (external and internal) flow into $Q_i$ during $V_k$ at rate $\hat\gamma_{i,k}$ and the fact that the length of $V_k$ in the fluid model is equal to $\hat\rho_kc$, one can show that the fluid limit of the mean amount of work in $Q_i$ at the beginning of  a visit to $Q_j$ is
\begin{equation}
\sum_{k=i}^{j-1} \hat\rho_k \hat\gamma_{i,k}\btilde_i c \qquad\text{for }j=i+1,\dots,i+N.
\label{fluidworkatvisitbeginnings}
 \end{equation}
This reduces to $\hat{\gamma}_i \btilde_i c$ for $j=i+N$.
We have used that in the fluid limit the fraction of time that the server is visiting $Q_j$ is $\hat{\rho}_j~(\jj)$.
Combining these observations, one can obtain the following expression for $\delta_i$, defined as the ratio of the fluid limit of the average amount of work at $Q_i$ and the length of a cycle (see Figure \ref{fluidHT}).
\begin{definition}
For $\ii$,
\begin{equation}
\delta_i
=
\frac{1}{2} \hat{\rho}_i \btilde_i ( \hat{\gamma}_i  + \hat{\rho}_i \hat{\gamma}_{i,i})
+
\sum_{j=i+1}^{i+N-1} \hat{\rho}_j
\left(
\frac{1}{2} \hat{\rho}_j \btilde_i\hat\gamma_{i,j}
+
\sum_{k=i}^{j-1} \hat{\rho}_k \btilde_i \hat\gamma_{i,k}
\right).\label{deltai}
\end{equation}\label{deltalemma}
\end{definition}
We introduce the notation $v_i$ for the average amount of work in $Q_i$,
\[
v_i=\delta_i c.
\]
As the \emph{total} inflow in all queues is equal to the total outflow per time unit, the total amount of work during a cycle remains constant at level
\[
v=\sum_{i=1}^N v_i=\sum_{i=1}^N\delta_i c=\delta c,
\]
where $\delta = \sum_{i=1}^N\delta_i$.

\paragraph{Amount of fluid.} We first study the number of fluid particles in each queue at various epochs during  the cycle. We denote by $X_{i,k}^{\textit{fluid}}$ the number of fluid particles in $Q_i$ at the \emph{beginning} of a visit to $Q_k$. Using \eqref{fluidworkatvisitbeginnings} and the fact that the amount of fluid in $Q_i$ is equal to the amount of work in $Q_i$ divided by $\btilde_i$, we obtain
\[
X_{i,k}^\textit{fluid}=\sum_{j=i}^{k-1} \hat\rho_j \hat\gamma_{i,j} c \qquad\text{for }k=i+1,\dots,i+N.
\]
Again, note that $X_{i,i+N}^\textit{fluid}$ can also be written as
\[
X_{i,i}^\textit{fluid}=\hat\gamma_i c.
\]
As a consequence, the amount of fluid at an arbitrary moment during $V_k$, denoted by $L^\textit{fluid}_{i,k}$, is uniformly distributed on the interval
\begin{equation}
\begin{array}{ll}
\left[
\sum_{j=i}^{k-1} \hat\rho_j \hat\gamma_{i,j} c,
\sum_{j=i}^{k} \hat\rho_j \hat\gamma_{i,j} c
\right]&\qquad \text{for }k=i+1,\dots,i+N-1,\\[2ex]
\left[
\hat\rho_i \hat\gamma_{i,i} c,
\hat\gamma_{i} c
\right]&\qquad  \text{for }k=i.
\end{array}
\label{LikDist}
\end{equation}
We can now obtain an expression for the distribution of the amount of fluid in $Q_i$ at an arbitrary epoch, by conditioning on the visit period:
\begin{align}
L^\textit{fluid}_{i} &\equaldist
L^\textit{fluid}_{i,k}
\qquad \textrm{ w.p. }\hat\rho_{k}, \label{LiFluid}
\end{align}
where $L^\textit{fluid}_{i,k}$ is uniformly distributed as in \eqref{LikDist}.
For notational reasons, we introduce the notation $\mathcal{L}^\textit{fluid}_{i}:=L^\textit{fluid}_{i}/c$ and $\mathcal{L}^\textit{fluid}_{i,k}:=L^\textit{fluid}_{i,k}/c$. One can consider $\mathcal{L}^\textit{fluid}_{i}$ as a standardised version of $L^\textit{fluid}_{i}$, not depending on the cycle length $c$.

\paragraph{Waiting times.}

For the fluid model under consideration we are interested in the waiting time distribution of an arbitrary fluid particle, internal or external. We define the waiting time as the time between the arrival in a queue, and the moment of departure from this queue (even if the particle is routed to another, or even the same queue). If we condition on the event that an arbitrary fluid particle arrives in $Q_i$ during $V_k$, its conditional waiting time consists of the residual part of $V_k$, the visit periods $V_{k+1}, \dots, V_{i-1}$, and the processing of the amount of fluid that has arrived in $Q_i$ during the elapsed part of the cycle, i.e., $V_i,\dots,V_{k-1}$ plus the elapsed part of $V_k$. Let $u_k$ be the fraction of $V_k$ that has elapsed at the arrival epoch of a fluid particle in $Q_i$. We denote by $W_{i,k}^{\textit{fluid}}(u_k)$, for $0\leq u_k \leq 1$, the conditional waiting time of an arbitrary fluid particle arriving in $Q_i$ during $V_k$, at the moment that a fraction $u_k$ of this visit period has elapsed. We have
\begin{equation}
W_{i,k}^{\textit{fluid}}(u_k) = \underbrace{\sum_{j=i-N}^{k-1}\hat\rho_jc\hat\gamma_{i,j}b_i+u_k\hat\rho_kc\hat\gamma_{i,k}b_i}_{\mathcal{P}_{i,k}(u_k)} + \underbrace{(1-u_k)\hat\rho_k c+\sum_{j=k+1}^{i-1}\hat\rho_j c}_{\mathcal{R}_{i,k}(u_k)},
\label{wikfluid}
\end{equation}
for $i=1,\dots,N$, $k=i-N,\dots,i-1$, and $0\leq u_k \leq 1$. As \eqref{wikfluid} indicates, we split $W_{i,k}^{\textit{fluid}}(u_k)$ into two parts, namely $\mathcal{P}_{i,k}(u_k)$ representing the waiting time due to the customers that arrived during the elapsed (Past) part of the cycle $c$, and $\mathcal{R}_{i,k}(u_k)$, which is the time until the next visit of the server to $Q_i$ (Residual cycle). This division into two parts turns out to be useful in the next paragraph, for computing the path-time distributions.

The waiting-time distribution follows after unconditioning. During $V_k$ fluid flows into $Q_i$ at rate $\hat\gamma_{i,k}$. Hence, the probability that an arbitrary fluid particle arrives during $V_k$, given that it arrives in $Q_i$, is $\pi_{i,k}:=\hat\gamma_{i,k}\hat\rho_k/\hat\gamma_i$. The elapsed fraction of the visit period $V_k$ is uniformly distributed on $[0,1]$. These two results yield the following expression for the waiting-time distribution of an arbitrary fluid particle in $Q_i$:
\begin{align}
W^\textit{fluid}_{i} &\equaldist
\mathcal{P}_{i,k}(U_k) + \mathcal{R}_{i,k}(U_k)  \qquad\textrm{ w.p. }\pi_{i,k}\nonumber\\
&= c \Big(1+\sum_{j=i-N}^{k-1}\hat\rho_j(\hat\gamma_{i,j}b_i-1)+U_k\hat\rho_k(\hat\gamma_{i,k}b_i-1)\Big)
\qquad \textrm{ w.p. }\pi_{i,k}, \label{WiFluid}
\end{align}
for $i=1,\dots,N$ and $k=i-N,\dots,i-1$. The random variables $U_1, \dots, U_N$ are independent and Uniform$[0,1]$ distributed. As before, we introduce the notation $\mathcal{W}^\textit{fluid}_{i}:=W^\textit{fluid}_{i}/c$ to represent a standardised version of $W^\textit{fluid}_{i}$, not depending on the cycle length $c$.

\paragraph{Path times.}

In this paragraph we derive the path-time (the total time spent in the system) distribution of customers - or, in this case, fluid particles - traversing a specific path through the network.  We denote the time spent in the system by a fluid particle traversing the path $Q_{i_1}, Q_{i_2}, \dots, Q_{i_M}$ by $W^\textit{fluid}_{i_1,i_2,\dots,i_M}$, where $i_k\in\{1,2,\dots,N\}$ for $k=1,2,\dots,M$ and $M \geq 1$. When considering path times, each fluid particle enters the system as an \emph{external} fluid particle in $Q_{i_1}$. Just like in the previous paragraph, we condition on the visit period and the length of the elapsed part of this visit period at its arrival epoch. Assume that a tagged fluid particle arrives in  $Q_{i_1}$ at the moment that the server is visiting $Q_k$ and a fraction $u_k$ of $V_k$ has elapsed $(0\leq u_k \leq 1)$. We denote its conditional path time as $W^\textit{fluid}_{i_1,i_2,\dots,i_M; k}(u_k)$. Its time spent in $Q_{i_1}$ is exactly $W_{{i_1},k}^{\textit{fluid}}(u_k)$. At the moment that the tagged fluid particle leaves $Q_{i_1}$ and is routed to $Q_{i_2}$, the elapsed time of $V_{i_1}$ is $\mathcal{P}_{i_1,k}(u_k)$. As a consequence, the elapsed \emph{fraction} of $V_{i_1}$ is $\mathcal{P}_{i_1,k}(u_k)/(\hat\rho_{i_1}c)$. This result leads to the following recursive equation for the conditional path time,
\[
W^\textit{fluid}_{i_1,i_2,\dots,i_M; k}(u_k) = W_{{i_1},k}^{\textit{fluid}}(u_k) + W^\textit{fluid}_{i_2,\dots,i_M; i_1}\left(\frac{\mathcal{P}_{i_1,k}(u_k)}{\hat\rho_{i_1}c}\right).
\]
We define $W^\textit{fluid}_{i_M; i_{M-1}}\left(\cdot\right) := W^\textit{fluid}_{i_M, i_{M-1}}\left(\cdot\right)$ to ensure that the recursion always ends.

The path-time distribution follows after unconditioning, noting that the probability that an \emph{external} fluid particle enters the system during $V_k$ is $\hat\rho_k$:
\begin{align}
W^\textit{fluid}_{i_1,i_2,\dots,i_M} &\equaldist
W^\textit{fluid}_{i_1,i_2,\dots,i_M; k}(U_k) \qquad\textrm{ w.p. }\hat\rho_{k},
\label{pathtimefluid}
\end{align}
for $i_1,i_2,\dots,i_M \in\{1,2,\dots,N\}$, and $k=1,2,\dots,N$. The random variables $U_1, \dots, U_N$ are independent and Uniform$[0,1]$ distributed.
Note that $W^\textit{fluid}_{i_1,i_2,\dots,i_M}/c$ does not depend on the cycle time $c$ anymore. Similar to the standardised waiting time, we now introduce the notation $\mathcal{W}^\textit{fluid}_{i_1,i_2,\dots,i_M} := W^\textit{fluid}_{i_1,i_2,\dots,i_M}/c$, which turns out to be useful later.

\subsection{Fluid model: Exhaustive service}

In this subsection we briefly discuss the fluid model when some of the queues are served exhaustively. Rather than repeating all of the analysis of the previous subsection, we mainly focus on the differences compared to a model with gated service. Note that changing the service discipline of a particular queue has no effects on any of the other queues.

\paragraph{Workload.}

Assume that the service discipline at a certain queue, say $Q_e$ with $e=1,\dots,N$, is exhaustive service. When comparing the fluid trajectory of the workload of $Q_e$ during the course of a cycle to the case with gated service (as depicted in Figure \ref{fluidHT}), the only change is that the trajectory needs to be moved downwards such that the queue is empty at the end of the visit period. More precisely, if $Q_e$ receives exhaustive service, then we define
\begin{equation}
\delta_e=\delta_e^\textit{gated}-\hat\rho_e\hat\gamma_{e,e}\btilde_e,\label{deltaiexhaustive}
\end{equation}
where $\delta_e^\textit{gated}$ is the value of $\delta_e$ given by \eqref{deltai} for the case that $Q_e$ would have received gated service.

\paragraph{Amount of fluid.} Since the amount of fluid in $Q_i$ is equal to the amount of work in $Q_i$ divided by $\btilde_i$, and since we have just established that the amount of work for an exhaustively served queue $Q_e$ is equal to the amount of work this queue would have under gated service, minus $\hat\rho_e\hat\gamma_{e,e}\btilde_e$, we directly obtain that the amount of fluid at an arbitrary moment during $V_k$, denoted by $L^\textit{fluid}_{e,k}$, is uniformly distributed on the interval
\begin{equation}
\begin{array}{ll}
\left[
\sum_{j=e+1}^{k-1} \hat\rho_j \hat\gamma_{i,j} c,
\sum_{j=e+1}^{k} \hat\rho_j \hat\gamma_{i,j} c
\right]&\qquad \text{for }k=e+1,\dots,e+N-1,\\[2ex]
\left[
0,
(\hat\gamma_{e} - \hat\rho_e\hat\gamma_{e,e})c
\right]&\qquad  \text{for }k=e.
\end{array}
\label{LikDistExh}
\end{equation}

\paragraph{Waiting times.}

Determining the waiting-time distribution of an arbitrary fluid particle also involves the same steps as in the gated case, except for fluid particles arriving in $Q_e$ during $V_e$, obviously.

\begin{equation*}
W_{e,k}^\textit{fluid}(u_k) =
\begin{cases}
\underbrace{(1-u_k)\hat\rho_k c+\sum_{j=k+1}^{e-1}\hat\rho_j c}_{\mathcal{R}_{e,k}(u_k)} + \underbrace{\sum_{j=e-N+1}^{k-1}\hat\rho_j c\hat\gamma_{e,j}b_e+u_k\hat\rho_k c\hat\gamma_{e,k}b_e}_{\mathcal{P}_{e,k}(u_k)}  &\quad \quad (k\neq e),\\
\underbrace{(1-u_e)(\hat\gamma_e c-\hat\gamma_{e,e}\hat\rho_e c)b_e}_{\mathcal{P}_{e,e}(u_e)},& \quad\quad(k=e).
\end{cases}
\end{equation*}
Note that $\mathcal{R}_{e,e}(u_e) = 0$ when the service in $Q_e$ is exhaustive.

\paragraph{Path times.}

We consider the conditional path time $W^\textit{fluid}_{e,i_2,\dots,i_M; k}(u_k)$, where $Q_{e}$ receives exhaustive service. As in the previous paragraph, we need to treat the case $e=k$ separately. A fluid particle arriving in $Q_e$ during $V_e$, given that a fraction $u_e$ of this visit period has elapsed, will wait for $W_{{e},e}^{\textit{fluid}}(u_e)$ before leaving this queue. This waiting time, which can also be denoted as $\mathcal{P}_{e,e}(u_e)$, is a fraction $\frac{\mathcal{P}_{e,e}(u_e)}{\hat\rho_{e}c}$ of the visit period $V_e$. This gives the following expression for the conditional path time:
\[
W^\textit{fluid}_{e,i_2,\dots,i_M; k}(u_k) =
\begin{cases}
W_{{e},k}^{\textit{fluid}}(u_k) + W^\textit{fluid}_{i_2,\dots,i_M; e}\left(\frac{\mathcal{P}_{e,k}(u_k)}{\hat\rho_{e}c}\right),
&\qquad (e \neq k),\\
W_{{e},e}^{\textit{fluid}}(u_e) + W^\textit{fluid}_{i_2,\dots,i_M; e}\left(u_e+\frac{\mathcal{P}_{e,e}(u_e)}{\hat\rho_{e}c}\right),
&\qquad (e = k).
\end{cases}
\]

\subsection{Original model: Gated service}\label{sect:originalmodelgated}

In this section we expand upon the HTAP for polling models by relating the limit processes of the total workload process and the workload of the individual queues. The main difference with polling models is that in the roving server network, (the direction of) the shifting of individual workloads is not only determined by which queue is being served but also by the internal routing of the customers. To this end, we return to the original model under HT conditions.
\paragraph{Workload.}
We denote by $V$ the total amount of work in the system at the start of a cycle starting, without loss of generality, at the beginning of visit period $V_1$. As far as the total amount of work is concerned, the system behaves like a polling system in heavy traffic with external customers bringing in an amount of work $\Btilde_i$ in $Q_i$, but with work shifting from one queue to another upon the service completion of a customer. For polling systems with general renewal arrivals the HT limit of the scaled total amount of work at the beginning of a cycle is conjectured by Olsen and Van der Mei \cite{olsenvdmei05}. 
An adaptation of the conjecture in \cite{olsenvdmei05}, in accordance with the proof for the case of Poisson arrivals in \cite{olsenvdmei03}, to our model leads to the following result.
\begin{conjecture}\label{conjecturescaledworkHT}
Define
\[
\sigma^2=\sum_{i=1}^N \hat\lambda_{i}\left(\V[\Btilde_i]+(\hat\lambda_i \btilde_i)^2\V[\hat A_{i}]\right),\qquad
\alpha = 2r\delta/\sigma^2,\qquad
\mu = 2/\sigma^2,
\]
with $\delta$ as defined in Definition \ref{deltalemma}. Then, for $\rho \uparrow 1$, $(1-\rho)V$ has a Gamma distribution with shape parameter $\alpha$ and rate parameter $\mu$.
\end{conjecture}
For more details we refer to \cite{olsenvdmei05} (who, in turn, refer to a result from \cite{coffman98}).

\paragraph{Cycle time.}
Subsequently, the diffusion limit of the \emph{total} workload process and the workload in the individual queues can be related using the HTAP. To this end, we start with the cycle-time distribution under HT scalings, which follows from Conjecture \ref{conjecturescaledworkHT} and the fluid analysis carried out in the first part of this section. The length of a cycle depends on the amount of work at the beginning of that cycle (which may be any arbitrarily chosen moment). Denote by $C(x)$ the length of a cycle, given that a total amount of $x$ work is present at its beginning. In steady state, we have the following relation
\begin{equation}
\delta C(x) = x ,\label{relationCandV}
\end{equation}
where $\delta$ is defined as in Definition \ref{deltalemma}. Hence, given an amount of work $x$, the cycle time is $C(x)=x/\delta$.
In the fluid model, all cycles have the same length. However, in the stochastic model, the cycle lengths are random. Note that, although the cycle-time distribution depends on the service disciplines and the chosen starting point of the cycle, the \emph{mean} cycle time is always $\E[C_i]=\E[C]=r/(1-\rho)$ (cf. \cite{sidi2}).
We are now ready to formulate the second conjecture, concerning the limiting distribution of the scaled length-biased cycle time.

\begin{conjecture}\label{conjecturescaledcycleHT}
For $\rho \uparrow 1$, we find that $(1-\rho){C_i}$ converges in distribution to a random variable having a Gamma distribution with shape parameter $\alpha$ and rate parameter $\delta\mu$.
\end{conjecture}

\paragraph{Queue lengths.} Given the cycle-time distribution, we can finally find the scaled queue-length distributions under HT conditions.
We use the fluid analysis, in combination with the conjectured cycle time distribution, to find the limiting distribution of the scaled queue lengths. In the fluid analysis the cycle time had a fixed length $c$. Due to the HTAP we can replace the constant cycle time from the fluid analysis by the random variable $\bm{C_i}$, the scaled \emph{length-biased} cycle time. If a random variable $X$ has a Gamma distribution with parameters $a$ and $b$, its length-biased equivalent $\bm{X}$ has a Gamma distribution with parameters $a+1$ and $b$. Obviously, the replacement of $c$ by $\bm{C_i}$ can only be carried out because of the independence between the length of the cycle time and the uniformly distributed random variables appearing in \eqref{LiFluid}. The following conjecture summarises this result. A theorem and proof for the scaled queue length distribution under the assumption of \emph{Poisson arrivals} can be found in the Appendix.

\begin{conjecture}\label{conjecturescaledqueuelengthHT}
As $\rho \uparrow 1$, the scaled queue length $(1-\rho)L_{i}$ converges in distribution to the product of two independent random variables. The first has the same distribution as $\mathcal{L}^\textit{fluid}_{i}$, and the second random variable $\Gamma$ has the same distribution as the limiting distribution of the scaled length-biased cycle time, $(1-\rho)\bm{C_i}$. For $i=1,\dots,N; k=i,\dots,i+N-1$, and $\rho\uparrow 1$,
\begin{equation}
(1-\rho)L_{i} \limdist \Gamma \times \mathcal{L}^\textit{fluid}_{i,k}\qquad \textrm{ w.p. }\hat\rho_{k},
\label{LiHT}\\
\end{equation}
where $\Gamma$ is a random variable having a Gamma distribution with parameters $\alpha+1$ and $\delta\mu$, and $\mathcal{L}^\textit{fluid}_{i,k}$ is the ``standardised'' number of a type-$i$ particles in the fluid model during $V_k$, introduced below \eqref{LiFluid}. The random variables $\Gamma$ and $\mathcal{L}^\textit{fluid}_{i,k}$ are independent.
\end{conjecture}

\paragraph{Waiting times.} The distributions of the scaled waiting times are obtained in a similar manner, exploiting the HTAP to combine the results from the fluid analysis and the scaled, length-biased cycle-time distribution, which leads to the following conjecture.

\begin{conjecture}\label{conjecturescaleddelayHT}
As $\rho \uparrow 1$, the scaled waiting time $(1-\rho)W_{i}$ converges in distribution to the product of two independent random variables. The first has the same distribution as $\mathcal{W}^\textit{fluid}_{i}$, and the second random variable $\Gamma$ has the same distribution as the limiting distribution of the scaled length-biased cycle time, $(1-\rho)\bm{C_i}$. For $i=1,\dots,N; k=i-N,\dots,i-1$, and $\rho\uparrow 1$,
\begin{equation}
(1-\rho)W_{i} \limdist \Gamma \times \mathcal{W}^\textit{fluid}_{i},
\label{WiHT}\\
\end{equation}
where $\Gamma$ is a random variable having a Gamma distribution with parameters $\alpha+1$ and $\delta\mu$, and $\mathcal{W}^\textit{fluid}_{i}$ is the ``standardised'' waiting time of a type-$i$ particle in the fluid model, introduced below \eqref{WiFluid}. The two random variables $\Gamma$ and $\mathcal{W}^\textit{fluid}_{i}$ are independent.
\end{conjecture}

The (HT limit of the) \emph{mean} waiting time of an arbitrary customer in $Q_i$ obviously follows from \eqref{WiHT}, but an easier way to find it, is by application of Little's Law to the mean queue length at $Q_i$, which is simply the mean amount of work in $Q_i$ divided by the mean total service time.
\noindent\begin{corollary}
For $\ii$,
\begin{equation}
(1-\rho)\E[W_i] \rightarrow \left(r+\frac{\sigma^2}{2\delta}\right)\frac{\delta_i}{\hat\gamma_i\btilde_i}, \qquad (\rho\uparrow1).
\label{EWht}
\end{equation}
\end{corollary}

\paragraph{Path times.}

The derivation of the limiting distribution of the scaled path times proceeds along the exact same lines, starting with the fluid model. Due to the HTAP, we may again replace the constant cycle time $c$ from the fluid analysis by the random variable $\bm{C_i}$ to obtain the limiting distribution of the scaled path times. The conjecture below summarises this result, which is consistent with the heavy-traffic snapshot principle \cite{whittsnapshot}.

\begin{conjecture}\label{conjecturescaledpathtimesHT}
As $\rho \uparrow 1$, the scaled path time $(1-\rho)W_{i_1,i_2,\dots,i_M}$ converges in distribution to the product of a random variable having the same distribution as $\mathcal{W}^\textit{fluid}_{i_1,i_2,\dots,i_M}$ and a random variable $\Gamma$ having the same distribution as the limiting distribution of the scaled length-biased cycle time, $(1-\rho)\bm{C_i}$. For $i,k=1,\dots,N$, and $\rho\uparrow 1$,

\begin{equation}
(1-\rho)W_{i_1,i_2,\dots,i_M} \limdist \Gamma \times \mathcal{W}^\textit{fluid}_{i_1,i_2,\dots,i_M},\\
\end{equation}
where $\Gamma$ is a random variable having a Gamma distribution with parameters $\alpha+1$ and $\delta\mu$, and the distribution of $\mathcal{W}^\textit{fluid}_{i_1,i_2,\dots,i_M}$ is given by \eqref{pathtimefluid}.
\end{conjecture}

\subsection{Original model: Exhaustive service}

In Subsection \ref{sect:originalmodelgated} we have used the HTAP to easily derive the limiting distributions of the scaled workload at cycle beginnings, the cycle times, the waiting times, and path times. This principle can be used in the exact same manner when some of the queues receive exhaustive service. Therefore, we will not repeat all these steps to summarise the results for systems with exhaustive service. Basically, one only has to take the expressions from the fluid model, and replace the constant cycle time $c$ by a random variable with the same limiting distribution as $(1-\rho)\bm{C_i}$, i.e., a Gamma distribution with parameters $\alpha+1$ and $\delta\mu$ (as defined earlier in this section, taking $\delta:=\delta_e$ as defined in \eqref{deltaiexhaustive} for exhaustive service).

\subsection{Poisson arrivals}

In the current paper we have derived the system behaviour under heavy traffic for systems with general renewal arrival processes based on the partially conjectured HTAP. Van der Mei \cite{RvdM_QUESTA} has developed a unifying framework to derive rigorous proofs of the heavy-traffic behaviour of branching-type polling models with (compound) \textit{Poisson} arrivals. By applying this stepwise approach in conjunction with the results of the previous section to the model under consideration, one can rigorously prove the HT asymptotics in queueing networks served by a single shared server under the assumption of Poisson arrivals. In the Appendix we provide such a proof for the queue-length distributions.

\subsection{Increasing setup times}
In HT the system reaches saturation due to an increase in the total utilisation $\rho$. However, the system might also get saturated due to an increase of the total switch-over time $r$. These two asymptotic regimes show, however, significantly different behaviour. In \cite{winandslargesetups07,winandslargesetupsbranching09} it was shown for polling systems that the scaled cycle and intervisit times converge in probability to deterministic quantities in the case that the (deterministic) switch-over times tend to infinity. One has to compare this with the Gamma distribution which is prevalent in the scaled cycle time in the diffusion limit of the present section. The results for polling systems with increasing switch-over times of \cite{winandslargesetups07,winandslargesetupsbranching09} can be extended to the setting of the current paper. That is, as a consequence of the scaled cycle time converging to a constant, a fluid limit is obtained implying that the scaled delay converges in distribution to a mixture of uniform distributions (cf. Formula \eqref{WiFluid}).

\section{Approximations for general traffic conditions}\label{sect:approx}

The HT distributions derived in the preceding section may be used directly as an approximation for the waiting-time and path-time distributions in non-heavy-traffic systems. However, they tend to perform poorly under low or moderate traffic. Therefore, in this section we combine these HT asymptotics with newly developed LT results leading to highly accurate approximations for the mean performance measures for the whole range of load values.   We will assess the accuracy of the resulting approximations in Section \ref{sect:numericalresults}.

\subsection{Waiting-time approximations}

In order to derive an approximation for the mean waiting time, we study the LT limit of $W_i$ which can be found by conditioning on the customer type (external or internally routed).
\noindent\begin{theorem}
For $\ii$ ,
\begin{equation}
W_i \limdist \begin{cases}
R^{\textit{res}} & \qquad\textrm{w.p. }{\lambda_i}/{\gamma_i}, \\
R_{j,i} &\qquad\textrm{w.p. }{\gamma_jp_{j,i}}/{\gamma_i},
\end{cases}
\qquad(\rho\downarrow0),
\label{Wlt}
\end{equation}
where $R^{\textit{res}}$ is a residual total switch-over time, with probability density function
\[
f_{R^\textit{res}}(t) = \frac{1-\P(R\leq t)}{\E[R]},
\]
and $R_{j,i} = R_j + R_{j+1}+\dots+R_{i-1}$ is the sum of the switch-over times between $Q_j$ and $Q_i$, which is a cyclic sum if $i\leq j$. Only if $i=j$ \emph{and} service in $Q_i$ is exhaustive, we have $R_{j,i} = 0$.
\\

\begin{proof}
In LT we ignore all $O(\rho)$ terms, which implies that we can consider a customer as being alone in the system. Equation \eqref{Wlt} can be interpreted as follows. An arbitrary customer in $Q_i$ has arrived from outside the network with probability ${\lambda_i}/{\gamma_i}$. In this case he has to wait for a residual total switch-over time $R^\textit{res}$. If a customer in $Q_i$ arrives after being served in another queue, say $Q_j$ (with probability ${\gamma_jp_{j,i}}/{\gamma_i}$), he has to wait for the mean switch-over times $R_j,\dots,R_{i-1}$.
\end{proof}
\end{theorem}

The LT limit of the \emph{mean} waiting time directly follows from \eqref{Wlt}:
\begin{equation}
\E[W_i] \rightarrow \frac{\lambda_i}{\gamma_i}\frac{r^{(2)}}{2r} + \sum_{j=i^*}^{i-1}\frac{\gamma_jp_{j,i}}{\gamma_i}\sum_{k=j}^{i-1}r_k, \qquad (\rho\downarrow0),
\label{EWlt}
\end{equation}
where $i^* = i-N$ if $Q_i$ has gated service, and $i^* = i-N+1$ if $Q_i$ has exhaustive service.
Subsequently, we construct an interpolation between the LT and HT limits that can be used as an approximation for the mean waiting time for arbitrary $\rho$. For $\ii$,
\begin{equation}
\E[W_i^{\textit{approx}}]=\frac{w_i^{\textit{LT}} + (w_i^{\textit{HT}}-w_i^{\textit{LT}})\rho}{1-\rho},\label{EWapprox}
\end{equation}
where $w_i^{\textit{LT}}$ and $w_i^{\textit{HT}}$ are the LT and HT limits of the mean waiting time respectively, as given in \eqref{EWlt} and \eqref{EWht}. Because of the way $\E[W_i^\textit{approx}]$ is constructed, it has the nice properties that it is exact as $\rho\downarrow0$ and $\rho\uparrow1$. Furthermore, if we have Poisson arrivals, it satisfies a so-called pseudo-conservation law for the mean waiting times, which is derived in \cite{sidi2}. This implies that the $\E[W_i^\textit{approx}]$ yields exact results for symmetric (and, hence, single-queue) systems. Finally, it can be shown that this approximation is exact in the limiting case of deterministic set up times that tend to infinity (see \cite{winandslargesetups07,winandslargesetupsbranching09}).

The astute reader has already noticed that the LT result (\ref{EWlt}) is  a first-order Taylor expansion of the mean waiting time at $\rho=0$, which can be naturally extended with the $m^{th}$ derivatives of the mean waiting time with respect to $\rho$ at $\rho=0$. Together with the HT limit one has $m+1$ pieces of information, which can be used to construct an $(m+1)^{th}$ degree polynomial interpolation (cf. \cite{boonapprox2011}). Therefore, it is not inconceivable that the approximation can be refined further, but since the primary goal of this paper has been the derivation of the path-time time distributions under HT conditions such refinements are beyond the scope of the paper. Moreover, the  presented first-order polynomial interpolation is already quite accurate as can be seen in the numerical evaluation.

\subsection{Path-time approximations}

The principle used to determine waiting-time approximations, can be applied to path times in the exact same manner. For reasons of compactness, we will not present the details in this paper. Almost all of the required ingredients to develop a path-time approximation have been discussed. The only missing piece of information is the LT limit of the mean path-time. However, it can easily be found given the LT limit of the mean waiting time \eqref{EWlt}:
\begin{equation}
\E[W_{i_1,i_2,\dots,i_M}] \rightarrow \frac{r^{(2)}}{2r} + b_{i_1} + \sum_{j=2}^M \left(r_{i_{j-1},i_j} + b_{i_j}\right),\qquad (\rho\downarrow0),
\label{EPathlt}
\end{equation}
where $r_{j,i} = \E[R_{j,i}]$, with $R_{j,i}$ as defined in \eqref{Wlt}.

In the current section we have focussed on an interpolation of the \textit{mean} waiting time and path times for the ease of presentation. It goes without saying that following the same recipe one could derive an identical interpolation approximation for higher moments as well and, subsequently, fit a phase-type distribution for the complete distribution. Alternatively, one could follow the idea of \cite{dorsman2011}, in which a redined HT distribution is derived as an approximation of the complete distribution for arbitrary loads. However, due to the strong correlation between the waiting times of a customer within a specific path, we have observed in numerical tests that the accuracy of this approximation technique is not as high as for standard polling models. We would like to end with noting that for the mean path times these correlations obviously do no play a role, whence the developed approximation for the mean figures gives excellent results as illustrated in the next section.

\section{Numerical evaluation}\label{sect:numericalresults}

In the current section we present four practical cases from completely different application areas, indicating the versatility of the studied roving server network and showing the practical usage of the developed  asymptotics and approximations. Moreover, these cases clearly illsutrate the fact that the path time is the most important performance measures in most practical applications. Special cases of the network  are, for example, standard polling systems \cite{takagi3}, tandem queues \cite{nair,taube}, multi-stage queueing models with parallel queues \cite{katayama}, feedback vacation queues \cite{boxmayechiali97, takine}, symmetric feedback polling systems \cite{takagifeedback,takine}, systems with a waiting room \cite{alineuts84,takacsfeedback77}. We would like to remind the reader that the only practical alternative to our expressions for a complete performance analysis of the studied network is simulation.

\subsection{Example 1: A production system with rework}

The first application is a multiproduct system with random yields introduced by \cite{Grasman1}, which is a stylised practical case of a producer of plastic bumpers for automobiles. Every item is produced in a production queue and is defect with probability $p$; a defect item has to be reproduced in the same production queue. However, defect items are first routed to a temporary storage queue, which is served immediately after the current queue, before they are routed back to the original queue. The combination of exhaustively served production and storage queues implies that newly arriving items are served during the current cycle, whereas defect items will be reproduced in the next cycle. There are no switch-over time and service time required for these storage queues, implying that the storage queues do not contribute to the utilisation of the system.

We consider a system consisting of five production queues with relative loads $(0.1, 0.2, 0.2, 0.2, 0.3)$, Poisson arrivals, exponentially distributed service times with mean 1, and switch-over times with constant value $5$. A typical feature of these type of production systems is that the switch-over times are relatively large compared to the processing times, in contrast to the next numerical example that we discuss. We have additional five storage queues to which the defect items are routed. Each item is defect with probability $0.25$. We run simulations for various values of $\rho$, and we compare the results with the limiting HT distributions. For each value of $\rho$ we run 100 simulations, each of length $10^8$. These settings yield extremely accurate estimates for the probability distribution functions.

\paragraph{Convergence to the HT limit.} In this example, mostly due to the relatively large switch-over times, the (scaled) path-time distributions converge relatively fast to their HT limits. This is illustrated in Table \ref{convergencetoHT} and in Figure \ref{convergencetoHTfigs}. Table \ref{convergencetoHT} depicts
the means, standard deviations and tail probabilities of the (scaled) path-time distributions for nine paths. Only production queues are included in the path names. For example, path $2\rightarrow2\rightarrow2$ actually includes two visits to the storage queue of production queue 2. The tail probabilities are given for path lengths of 20, 50, 80 if the corresponding path consists of respectively 1, 2, and 3 visits to a queue. The results for $\rho<1$ are obtained using simulation, and the results for $\rho=1$ are obtained using the theory presented in Section \ref{sect:htap}. In Table \ref{convergencetoHT} and Figure \ref{convergencetoHTfigs}, it can be seen that already for $\rho=0.8$ the scaled path-time distributions are close to the limiting distributions.

\begin{table}[h]
\[
\begin{array}{|c|ccc|ccc|ccc|}
\hline
\multicolumn{10}{|c|}{\cellcolor[gray]{0.3}\color{white}\textrm{Scaled path times: Means}}\\
\hline
\rho & 1 & 1\rightarrow1 & 1\rightarrow1\rightarrow1 & 2 & 2\rightarrow2 & 2\rightarrow2\rightarrow2 & 5 & 5\rightarrow5 & 5\rightarrow5\rightarrow5\\
\hline
0.1 & 13.40 & 39.29 & 65.27 & 13.32 & 39.08 & 65.07 & 13.21 & 38.88 & 64.76 \\
0.3 &  13.22 & 38.89 & 64.84 & 12.92 & 38.26 & 64.14 & 12.64 & 37.62 & 63.37 \\
0.5 &  13.00 & 38.58 & 64.47 & 12.53 & 37.50 & 63.22 & 12.03 & 36.41 & 62.04 \\
0.7 & 12.79 & 38.30 & 64.19 & 12.10 & 36.77 & 62.51 & 11.40 & 35.24 & 60.82 \\
0.8 &  12.67 & 38.18 & 64.10 & 11.88 & 36.42 & 62.18 & 11.08 & 34.65 & 60.23 \\
0.9 &  12.55 & 38.08 & 64.03 & 11.65 & 36.07 & 61.86 & 10.74 & 34.06 & 59.67 \\
0.95 &  12.48 & 38.03 & 64.02 & 11.53 & 35.90 & 61.73 & 10.57 & 33.76 & 59.40\\
\hline
1.00 & 12.42 & 38.00 & 64.53 & 11.41 & 35.74 & 61.95 & 10.40 & 33.47 & 59.37 \\
\hline
\multicolumn{10}{c}{\ }\\
\hline
\multicolumn{10}{|c|}{\cellcolor[gray]{0.3}\color{white}\textrm{Scaled path times: Standard deviations}}\\
\hline
\rho & 1 & 1\rightarrow1 & 1\rightarrow1\rightarrow1 & 2 & 2\rightarrow2 & 2\rightarrow2\rightarrow2 & 5 & 5\rightarrow5 & 5\rightarrow5\rightarrow5\\
\hline
0.1 &   7.39 & 7.86 & 8.40 & 7.33 & 7.83 & 8.34 & 7.25 & 7.78 & 8.31 \\
0.3 &   7.59 & 8.91 & 10.29 & 7.39 & 8.76 & 10.16 & 7.19 & 8.62 & 10.05 \\
0.5 &   7.77 & 9.94 & 12.24 & 7.44 & 9.70 & 12.02 & 7.09 & 9.45 & 11.80 \\
0.7 &    7.93 & 11.02 & 14.22 & 7.46 & 10.66 & 13.98 & 6.98 & 10.29 & 13.71 \\
0.8 &    7.99 & 11.57 & 15.27 & 7.46 & 11.15 & 15.01 & 6.92 & 10.71 & 14.70 \\
0.9 &    8.05 & 12.12 & 16.33 & 7.45 & 11.63 & 16.06 & 6.85 & 11.14 & 15.72 \\
0.95 &   8.07 & 12.41 & 16.89 & 7.45 & 11.89 & 16.61 & 6.82 & 11.37 & 16.27 \\
\hline
1.00 & 8.09 & 12.69 & 18.70 & 7.44 & 12.13 & 18.07 & 6.78 & 11.58 & 17.45 \\
\hline
\multicolumn{10}{c}{\ }\\
\hline
\multicolumn{10}{|c|}{\cellcolor[gray]{0.3}\color{white}\textrm{Scaled path times: Tail probabilities}}\\
\hline
\rho & 1 & 1\rightarrow1 & 1\rightarrow1\rightarrow1 & 2 & 2\rightarrow2 & 2\rightarrow2\rightarrow2 & 5 & 5\rightarrow5 & 5\rightarrow5\rightarrow5\\
\hline
\rho     &  20 & 50 & 80 &  20 & 50 & 80&  20 & 50 & 80\\
\hline
0.8 &   0.191 & 0.154 & 0.159 & 0.152 & 0.120 & 0.130 & 0.113 & 0.089 & 0.103 \\
0.9 &   0.187 & 0.161 & 0.175 & 0.144 & 0.123 & 0.142 & 0.103 & 0.090 & 0.113 \\
0.95 &  0.186 & 0.165 & 0.184 & 0.141 & 0.125 & 0.149 & 0.098 & 0.090 & 0.118 \\
\hline
1.00 &   0.184 & 0.168 & 0.192 & 0.138 & 0.126 & 0.156 & 0.093 & 0.091 & 0.123 \\
\hline
\end{array}
\]
\caption{Numerical results for Example 1. 
}
\label{convergencetoHT}
\end{table}

\begin{figure}[h]
\begin{center}
\parbox{0.7\textwidth}{\centering
\includegraphics[width=\linewidth]{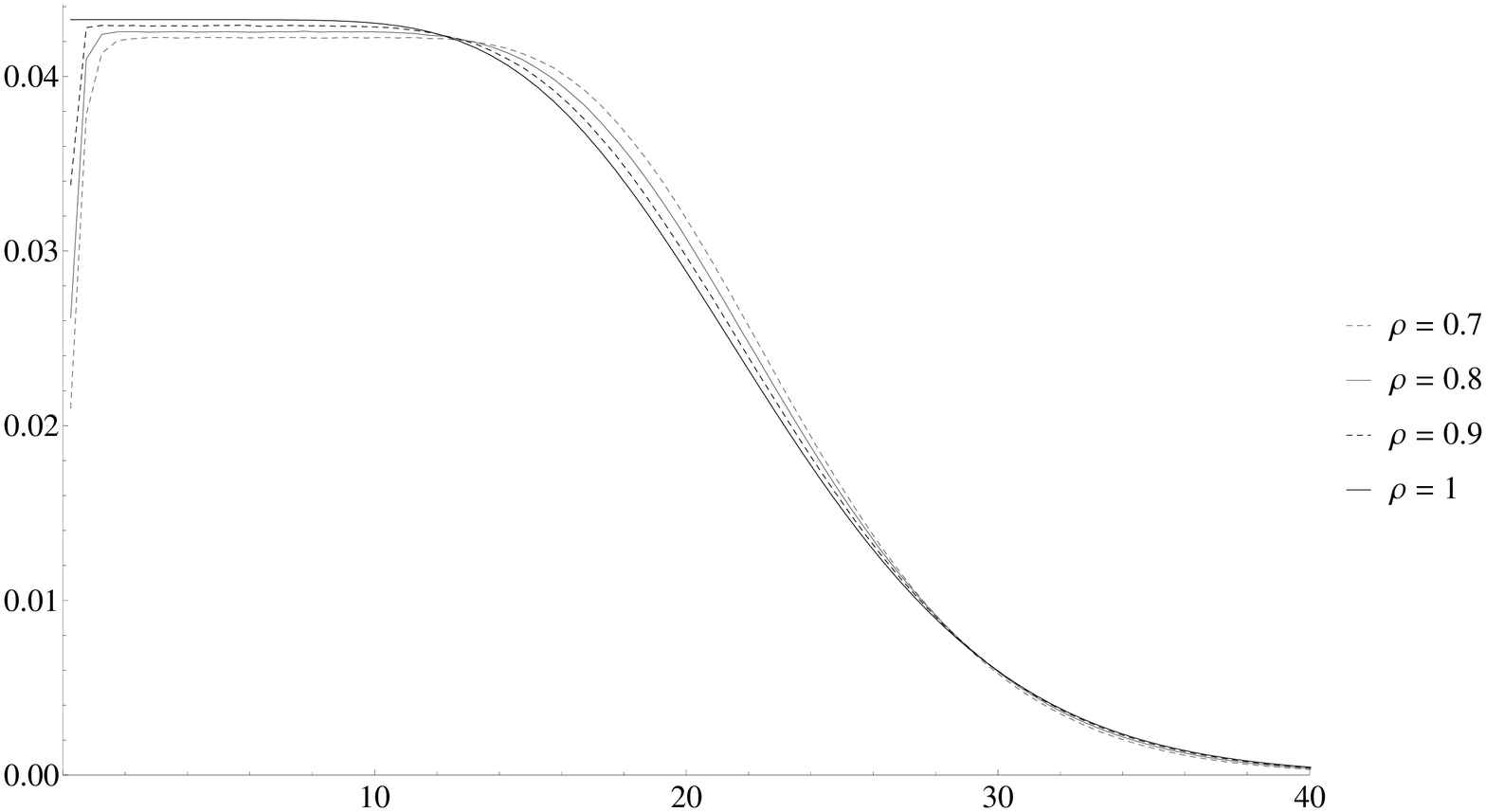}\\
$1$
}\\
\parbox{0.7\textwidth}{\centering
\includegraphics[width=\linewidth]{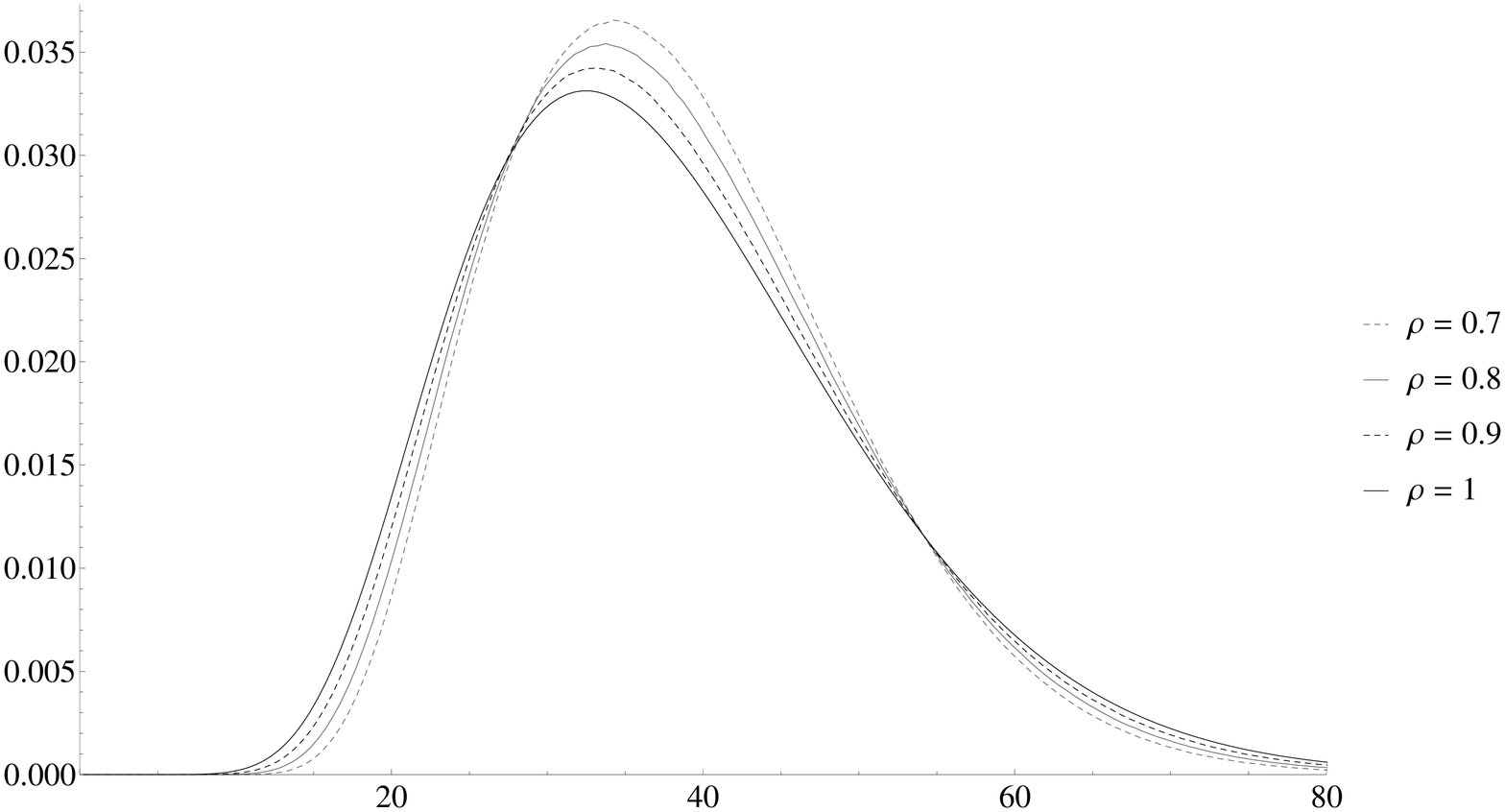}\\
$1\rightarrow1$
}\\
\parbox{0.7\textwidth}{\centering
\includegraphics[width=\linewidth]{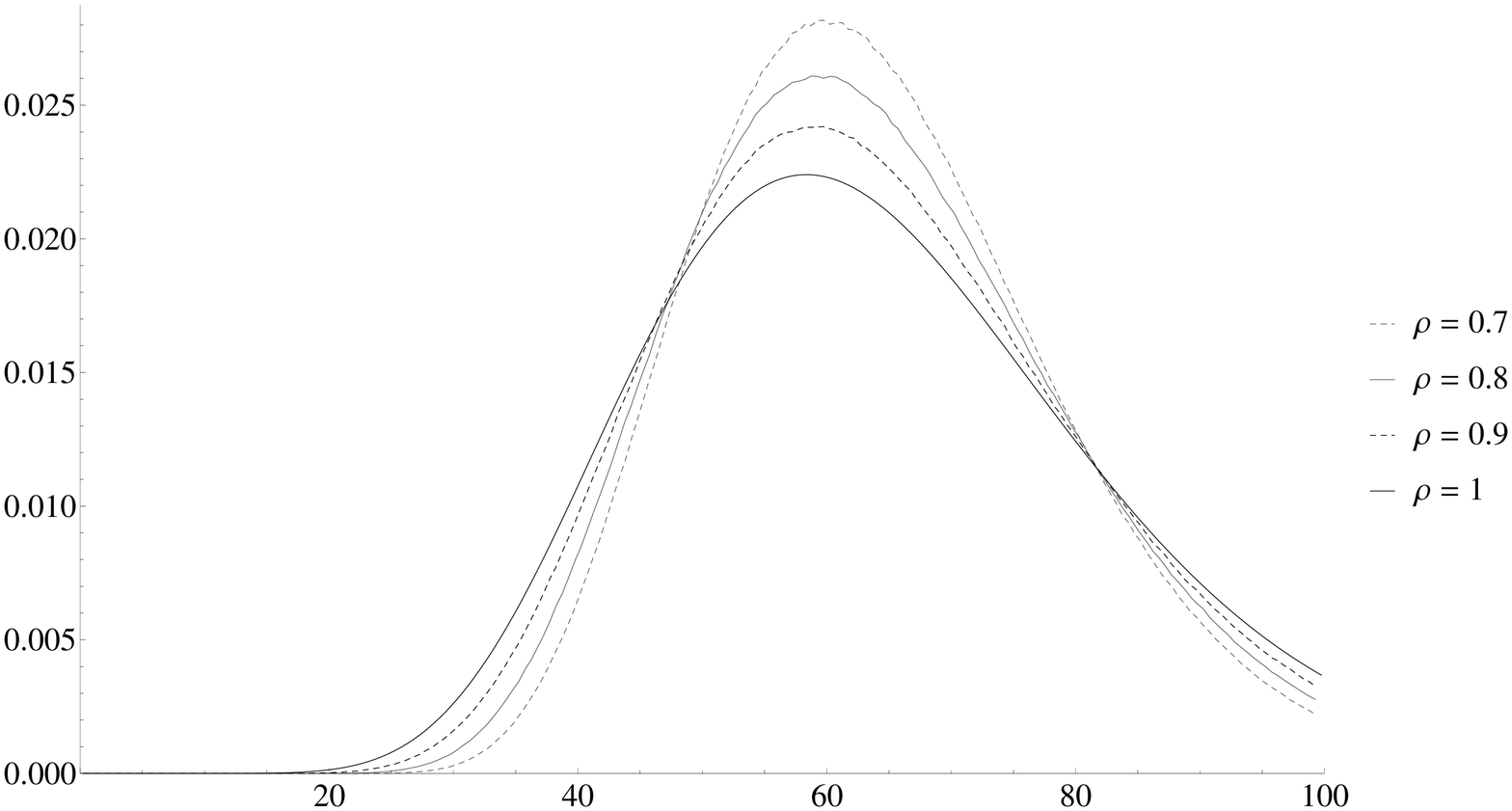}\\
$1\rightarrow1\rightarrow1$
}
\end{center}
\caption{The probability densities for three selected paths of the first numerical example. The results for $\rho<1$ are obtained using simulation. The results for $\rho=1$ are obtained using the analytical results from Section \ref{sect:htap}.}
\label{convergencetoHTfigs}
\end{figure}

\paragraph{Accuracy of the approximation.} Table \ref{accuracyApprox} shows the approximated means  using the  path-time approximation discussed in Section \ref{sect:approx}. For nine paths, the means  of the (scaled) path-time approximations are given. Only production queues are included in the path names. For example, path $2\rightarrow2\rightarrow2$ actually includes two visits to the storage queue of production queue 2. Note that we have chosen to display the values for the \emph{scaled} path times $(1-\rho)W_{i_1,i_2,\dots,i_M}$, because now they can be compared directly to the simulated values in Table \ref{convergencetoHT}. The results show that the accuracy is very high for all values of the load and for all paths.

\begin{table}[h]
\begin{center}
\[
\begin{array}{|c|ccc|ccc|ccc|}
\hline
\multicolumn{10}{|c|}{\cellcolor[gray]{0.3}\color{white}\textrm{Approximated scaled path times: Means}}\\
\hline
\rho & 1 & 1\rightarrow1 & 1\rightarrow1\rightarrow1 & 2 & 2\rightarrow2 & 2\rightarrow2\rightarrow2 & 5 & 5\rightarrow5 & 5\rightarrow5\rightarrow5\\
\hline
0.1 &  13.39 & 39.35 & 65.40 & 13.29 & 39.12 & 65.14 & 13.19 & 38.90 & 64.89 \\
0.3 &  13.17 & 39.05 & 65.21 & 12.87 & 38.37 & 64.43 & 12.57 & 37.69 & 63.66 \\
0.5 &  12.96 & 38.75 & 65.01 & 12.45 & 37.62 & 63.72 & 11.95 & 36.49 & 62.43 \\
0.7 &  12.74 & 38.45 & 64.82 & 12.04 & 36.86 & 63.01 & 11.33 & 35.28 & 61.21 \\
0.8 &  12.63 & 38.30 & 64.72 & 11.83 & 36.49 & 62.66 & 11.02 & 34.68 & 60.60 \\
0.9 &  12.52 & 38.15 & 64.63 & 11.62 & 36.11 & 62.30 & 10.71 & 34.07 & 59.98 \\
0.95 & 12.47 & 38.08 & 64.58 & 11.51 & 35.92 & 62.13 & 10.56 & 33.77 & 59.68 \\
\hline
1.00 & 12.42 & 38.00 & 64.53 & 11.41 & 35.74 & 61.95 & 10.40 & 33.47 & 59.37 \\
\hline

\end{array}
\]
\end{center}
\caption{Accuracy of the approximation in Example 1.}
\label{accuracyApprox}
\end{table}

\subsection{Example 2: Tandem queues with parallel queues in the first stage.}

\begin{figure}[ht]
\begin{center}
\includegraphics[width=0.5\linewidth]{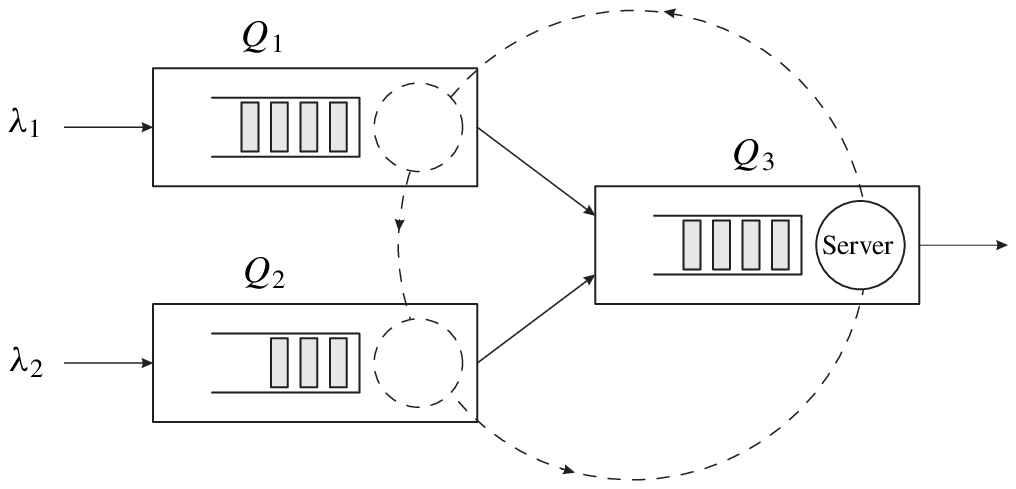}
\end{center}
\caption{Tandem queues with parallel queues in the first stage, as discussed in Example 2.}
\label{figureExample1}
\end{figure}

Secondly, we use an example that was introduced by Katayama \cite{katayama}, who studies call processing in packet switching systems composed of multi-processors. In this application, each processor has two kinds of buffers in parallel for receiving data packets from switching links and from lines and/or subscribers. During the first stage, called input processing, the processor polls both buffers for data packets and moves these packets to a queue in the second stage. During this second stage, called packet processing, the jobs are actually executed. This system can be modelled as a network consisting of three queues. Customers arrive at $Q_1$ and $Q_2$, and are routed to $Q_3$ after being served (see Figure \ref{figureExample1}). This tandem queueing model with parallel queues in the first stage is a special case of a roving server network. We simply put $p_{1,3}=p_{2,3}=p_{3,0}=1$ and all other $p_{i,j}$ are zero. We use the same values as in \cite{katayama}: Poisson arrivals with $\lambda_1=\lambda_2/10$, service times are deterministic with $b_1=b_2=1$, and $b_3=5$. The server serves the queues exhaustively, in cyclic order: 1, 2, 3, 1, \dots. The only difference with the model discussed in \cite{katayama} is that we introduce (deterministic) switch-over times $r_2=r_3=2$. We assume that no time is required to switch between the two queues in the first stage, so $r_1=0$.

\paragraph{Convergence to the HT limit.} For the two possible paths, the means, standard deviations and tail probabilities of the (scaled) path-time distributions are given in Table \ref{convergencetoHT2} and in Figure \ref{convergencetoHTfigs2}.
The tail probabilities in Table \ref{convergencetoHT2} are given for path lengths of 20. The results for $\rho<1$ are obtained using simulation, and the results for $\rho=1$ are obtained using the theory presented in Section \ref{sect:htap}.
It can be clearly observed that for loads larger than $0.8$ the scaled path-time distributions are almost identical to the limiting distributions.

\begin{table}[h]
\[
\begin{array}{|c|cc|}
\hline
\multicolumn{3}{|c|}{\cellcolor[gray]{0.3}\color{white}\textrm{Scaled path times: Means}}\\
\hline
\rho & 1\rightarrow3 & 2\rightarrow3\\
\hline
0.1 &  9.58 & 9.70 \\
0.3 &  8.69 & 9.12 \\
0.5 &  7.74 & 8.56 \\
0.7 &  6.75 & 8.01 \\
0.8 &  6.24 & 7.73 \\
0.9 &  5.72 & 7.47 \\
0.95 & 5.46 & 7.34 \\
\hline
1.00 & 5.20 & 7.20 \\
\hline
\multicolumn{3}{c}{\ }\\
\hline
\multicolumn{3}{|c|}{\cellcolor[gray]{0.3}\color{white}\textrm{Scaled path times: Standard deviations}}\\
\hline
\rho & 1\rightarrow3 & 2\rightarrow3\\
\hline
0.1 & 1.91 & 2.02  \\
0.3 &  2.75 & 3.04 \\
0.5 &  3.30 & 3.75 \\
0.7 &  3.74 & 4.30 \\
0.8 &  3.94 & 4.54 \\
0.9 &  4.13 & 4.76 \\
0.95 & 4.23 & 4.88 \\
\hline
1.00 & 4.32 & 4.97 \\
\hline
\multicolumn{3}{c}{\ }\\
\hline
\multicolumn{3}{|c|}{\cellcolor[gray]{0.3}\color{white}\textrm{Scaled path times: Tail probabilities}}\\
\hline
\rho & 1\rightarrow3 & 2\rightarrow3\\
\hline
0.1 & 0.002 & 0.002  \\
0.3 &  0.006 & 0.008 \\
0.5 &  0.008 & 0.013 \\
0.7 &  0.010 & 0.018 \\
0.8 &  0.010 & 0.020 \\
0.9 &  0.011 & 0.022 \\
0.95 & 0.011 & 0.023 \\
\hline
1.00 & 0.011 & 0.023\\
\hline
\end{array}
\]
\caption{Numerical results for Example 2.}
\label{convergencetoHT2}
\end{table}

\begin{figure}[h]
\begin{center}
\parbox{0.7\textwidth}{\centering
\includegraphics[width=\linewidth]{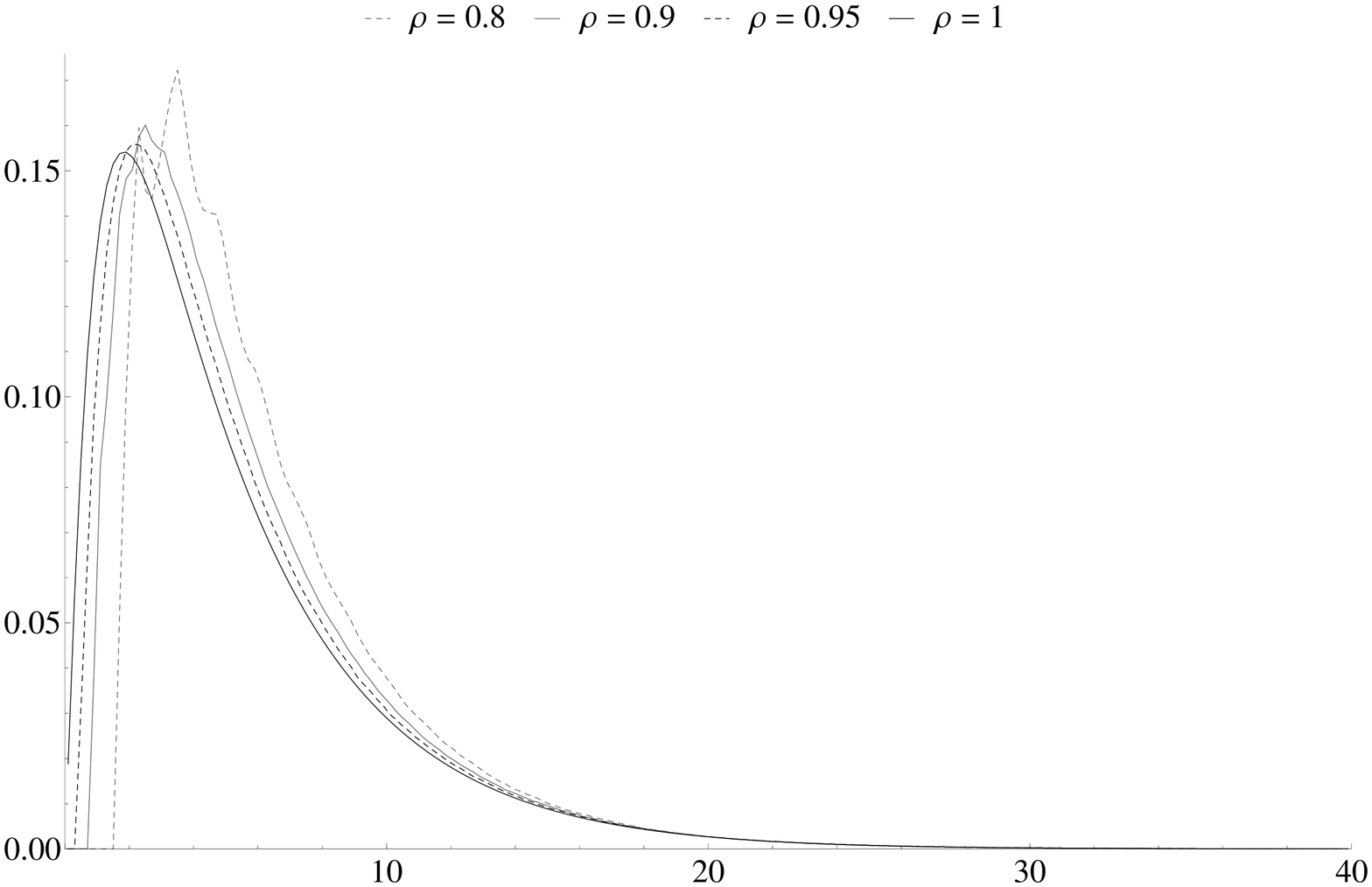}\\
$1\rightarrow3$
}\\
\parbox{0.7\textwidth}{\centering
\includegraphics[width=\linewidth]{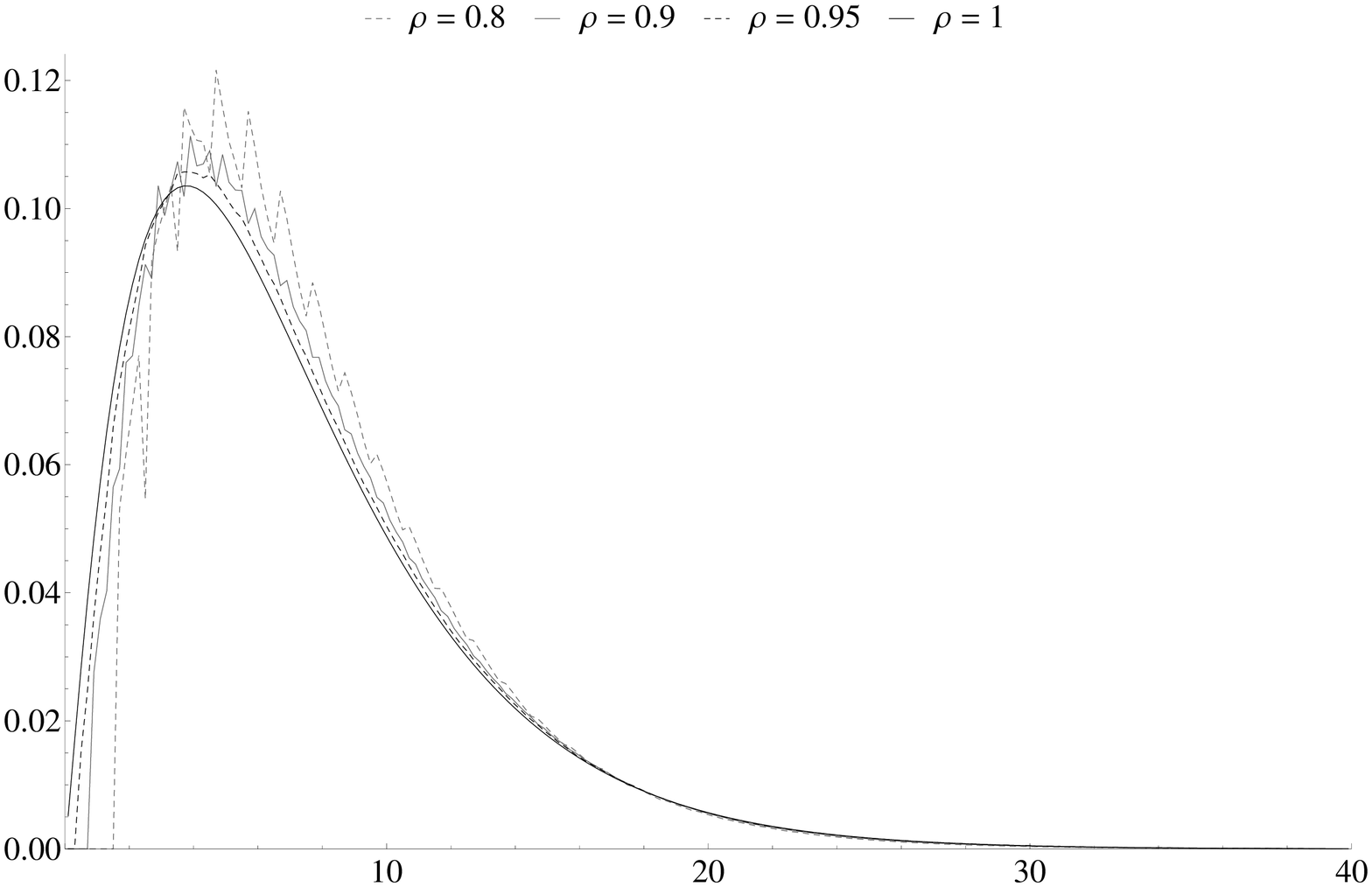}\\
$2\rightarrow3$
}
\end{center}
\caption{The probability densities for the two paths of the second numerical example. The results for $\rho<1$ are obtained using simulation. The results for $\rho=1$ are obtained using the analytical results from Section \ref{sect:htap}.}
\label{convergencetoHTfigs2}
\end{figure}

\paragraph{Accuracy of the approximation.} Table \ref{accuracyApprox2} shows the approximated means  using the  path-time approximation discussed in Section \ref{sect:approx} for both paths. Again the quality of the approximation is excellent for all possible loads and all paths: The relative approximation error is less than $2\%$ for all values of $\rho$.

We have also tested the accuracy of the approximation for different interarrival-time distributions, with squared coefficient of variation (SCV) equal to respectively $\frac12$ and 2. In the first case we have fitted a  mixed Erlang distribution, and in the second case a hyperexponential distribution. In these cases, the approximation still gives good results but slightly less accurate than for the case with Poisson arrivals: when the SCV equals 2, the relative approximation error is less than $4\%$ for all values of $\rho$. When the SCV is equal to $\frac12$, the relative error is less than $9\%$. As expected, the relative difference between the simulated and approximated values is biggest when $\rho$ lies between 0.3 and 0.7.

\begin{table}[h]
\begin{center}
\[
\begin{array}{|c|cc|}
\hline
\multicolumn{3}{|c|}{\cellcolor[gray]{0.3}\color{white}\textrm{Approximated mean scaled path times}}\\
\hline
\rho & 1\rightarrow3 & 2\rightarrow3\\
\hline
0.1 &   9.52 & 9.72 \\
0.3 &   8.56 & 9.16 \\
0.5 &   7.60 & 8.60 \\
0.7 &   6.64 & 8.04 \\
0.8 &   6.16 & 7.76 \\
0.9 &   5.68 & 7.48 \\
0.95 &  5.44 & 7.34 \\
\hline
1.00 & 5.20 & 7.20 \\
\hline
\end{array}
\]
\end{center}
\caption{Approximation results for Example 2. For the two paths, the means  of the (scaled) path-time approximations are given. }
\label{accuracyApprox2}
\end{table}

\subsection{Example 3: A file-server application}

This application is taken from Sidi et al. \cite{sidi2}, who consider a token ring network with one file-server and $K$ workstations, that transmit file requests to the file-server, which
in turn replies to the different stations by sending back the files they requested. The
performance measure of interest is the response time of a file request, which
is the time from the request generation by the workstation (being queued at that
time at the output queue of the station, awaiting its turn to be transmitted) until the
file arrives back at the station. We can model this system as a roving server network with $K+1$ queues. External customers arrive at $Q_1, \dots, Q_K$ and are routed to $Q_{K+1}$ after their service completion. Once served at $Q_{K+1}$ the customer leaves the system. We are interested in the path time $W_{i,K+1}$ for $i=1,\dots,K$. In our numerical example we take $N=K+1=11$, identical arrival rates at the $K$ workstations and no external arrivals at the file server, i.e. $\lambda_{K+1}=0$. Since arrival processes in this application area tend to exhibit high variation, we assume that the interarrival times follow a hyperexponential distribution with balanced means (see, e.g., \cite{tijms94}) and squared coefficient of variation equal to 4. The routing probabilities are all 0 except $p_{i,K+1}=1$ for $i=1,\dots,K$. The service times $B_1, \dots, B_K$ are all deterministic with value $0.1$, and the service times  at $Q_{K+1}$ are exponentially distributed with mean 1. The switch-over times in this kind of system are typically very small compared to the service times, so we take exponentially distributed switch-over times with mean $0.01$. The service discipline is gated at all queues.

We have chosen to highlight only a few typical results for this example, in Table \ref{resultsExample2} and Figure \ref{resultsExample2figs}, where
the means and standard deviations of the scaled path times are shown. The values for $\rho=0,9, 0.95, 0.98$ have been obtained by simulation, whereas the values for $\rho=1$ are obtained using the analytical results from Section \ref{sect:htap}.
The most typical feature is that  the densities of the waiting times and path times exhibit oscillating behaviour for smaller values of $\rho$ (see Figure \ref{resultsExample2figs}). Most interestingly, this is not caused by simulation inaccuracy, but by the fact that the path times densities are (almost completely) concentrated on a discrete set of values for smaller loads due to the combination of small switch-over times, \emph{deterministic} service times and \emph{deterministic} routing. This effect disappears when the loads increase, but as a consequence the convergence to the limiting HT distribution is not as fast as in the previous example. Nevertheless, despite the anomalous behaviour of the system's performance for small loads, the derived asymptotic turns out to be exact again.

\begin{table}[t]
\[
\begin{array}{|c|cccccccccc|}
\hline
\multicolumn{11}{|c|}{\cellcolor[gray]{0.3}\color{white}\textrm{Scaled path times: Means}}\\
\hline
\rho & 1\rightarrow 11 & 2\rightarrow11 & 3\rightarrow11 & 4\rightarrow11 & 5\rightarrow11 & 6\rightarrow11 & 7\rightarrow11 & 8\rightarrow11 & 9\rightarrow11 & 10\rightarrow11\\
\hline
0.9  &  1.25 & 1.38 & 1.51 & 1.64 & 1.77 & 1.89 & 2.02 & 2.14 & 2.26 & 2.37 \\
0.95 &   1.39 & 1.55 & 1.71 & 1.86 & 2.02 & 2.17 & 2.31 & 2.46 & 2.60 & 2.73 \\
0.98 &   1.47 & 1.65 & 1.82 & 2.00 & 2.17 & 2.33 & 2.50 & 2.66 & 2.81 & 2.96 \\
\hline
1.00    & 1.53 & 1.73 & 1.92 & 2.12 & 2.32 & 2.51 & 2.71 & 2.91 & 3.10 & 3.30 \\
\hline
\multicolumn{11}{c}{\ }\\
\hline
\multicolumn{11}{|c|}{\cellcolor[gray]{0.3}\color{white}\textrm{Scaled path times: Standard deviations}}\\
\hline
\rho & 1\rightarrow 11 & 2\rightarrow11 & 3\rightarrow11 & 4\rightarrow11 & 5\rightarrow11 & 6\rightarrow11 & 7\rightarrow11 & 8\rightarrow11 & 9\rightarrow11 & 10\rightarrow11\\
\hline
0.9  &  1.44 & 1.57 & 1.70 & 1.83 & 1.95 & 2.06 & 2.17 & 2.28 & 2.38 & 2.47 \\
0.95 &   1.58 & 1.72 & 1.85 & 1.98 & 2.11 & 2.23 & 2.35 & 2.45 & 2.56 & 2.65 \\
0.98 &   1.64 & 1.78 & 1.92 & 2.05 & 2.18 & 2.30 & 2.42 & 2.53 & 2.63 & 2.72 \\
\hline
1.00    &  1.74 & 1.90 & 2.08 & 2.25 & 2.43 & 2.61 & 2.79 & 2.97 & 3.16 & 3.34 \\
\hline
\end{array}
\]
\caption{Numerical results for Example 3. }
\label{resultsExample2}
\end{table}

\begin{figure}[h]
\begin{center}
\includegraphics[width=0.5\textwidth]{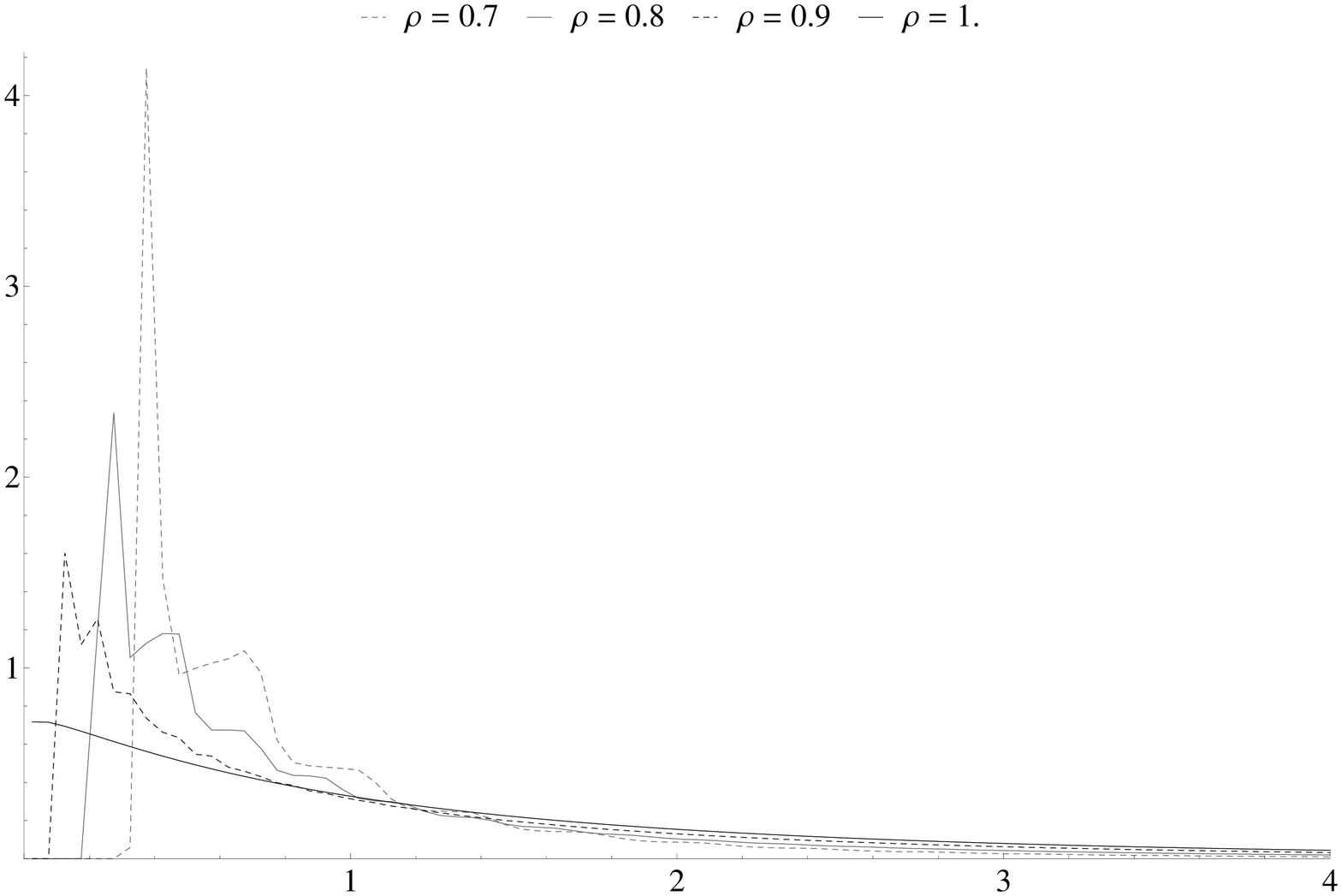}\\
$1\rightarrow 11$\\

\bigskip
\includegraphics[width=0.5\textwidth]{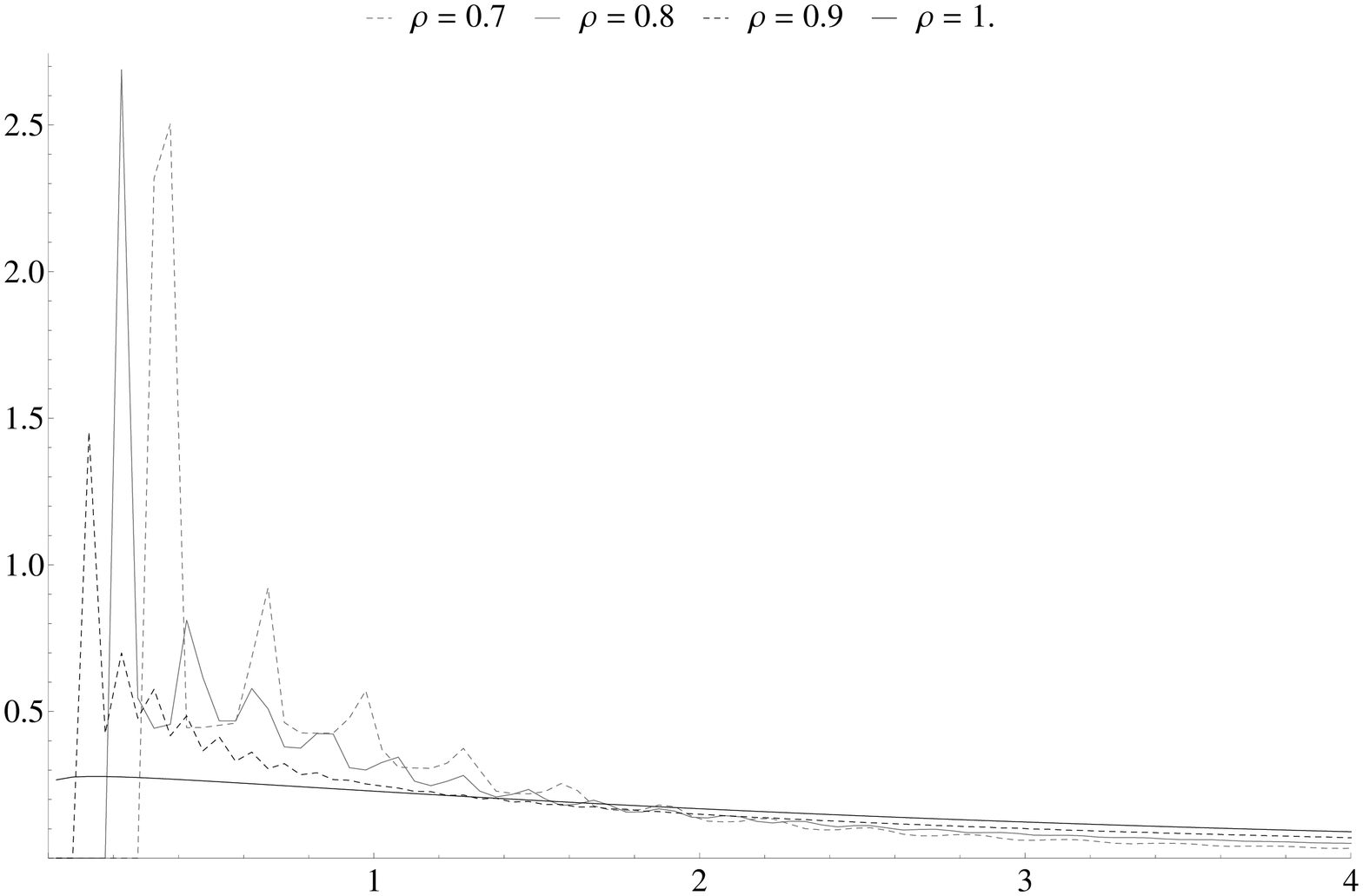}\\
 $10\rightarrow 11$\\
\end{center}
\caption{Numerical results for Example 3. The two figures show the densities of the scaled path-time distributions of the paths $1\rightarrow 11$ and $10\rightarrow 11$.}
\label{resultsExample2figs}
\end{figure}

\subsection{Example 4: A production system with a joint packaging queue and random yield}

In the fourth example we look at a production system with two types of products and a joint packaging queue and random yield. A set-up which is typically observed in many practical applications. Product type $A$ requires three operations, at queue $1$, queue $3$ and queue $5$, before joining the packaging queue $7$, whereas the other product type $B$ visits queue $2$, queue $4$ and queue $6$ before it is sent to the joint packaging queue.  Every production step fails with probability $p=0.01$, after which the product is sent to the beginning of the production process, i.e. queue $1$ for product $A$ and queue $2$ for product $B$. Packaging is, however, always successful. We assume that all queues are served exhaustively.

We study a system with external arrivals at queue $1$ and queue $2$ with intensity $0.12$ and $0.04$, respectively.  Furthermore we assume Erlang distributed arrivals with SCV $0.25$, Erlang distributed service times with means  1 (for the regular queues) and 2 (for the packaging queue) and SCV $0.5$, and switch-over times with constant value $5$.

For two possible paths, 
the simulated scaled path-time distributions are depicted in Figure \ref{resultsExample4figs}. The analytically computed limiting distribution at $\rho=1$ is included as well. As in Example 1, due to the relatively large switch-over times we see that for loads larger than $0.8$ the scaled path-time distributions almost coincide with the limiting distributions.

Lastly, we show an illustration how the performance of the system depends on the rework probability $p$. Therefore, we vary $p$ while keeping all the other parameters fixed. Note that increasing $p$ will also increase the total workload of the system. When considering the load of the system as a function of $p$, the stability condition can be rewritten to $p < 0.1558$. Figure \ref{example4fig2} shows the mean waiting time at queue $1$ as function of this probability; this figure clearly indicates that the control of the rework probability is crucial for obtaining satisfactory system performance. Moreover, it shows that the approximation for the mean waiting time in $Q_1$ is extremely accurate. This is as expected, since $p=0$ corresponds to a system with total load $\rho=0.8$, and we have concluded before that the approximation is very accurate for $\rho>0.8$. We emphasise that the approximations given in this paper are closed-form expressions, once all parameter values have been substituted. For example, the approximation for $\E[W_1]$ simplifies to \eqref{ewapproxp}. As a consequence, the approximation, is very suitable for optimisation purposes.

\begin{align}
\E[W_1^\textit{approx}]&=\left({1088 p^6-6374 p^5+15814 p^4-21412 p^3+16492 p^2-6406 p+670}\right)^{-1}\nonumber\\
&\times\left(192 p^{11}-2168 p^{10}+16940 p^9-99647 p^8+378503 p^7-880840 p^6+1234922 p^5\right.\nonumber\\
&\quad\left.-952526 p^4+202040 p^3+289849
   p^2-244631 p+52916\right).
   \label{ewapproxp}
\end{align}

\begin{figure}[h]
\begin{center}
\includegraphics[width=0.5\textwidth]{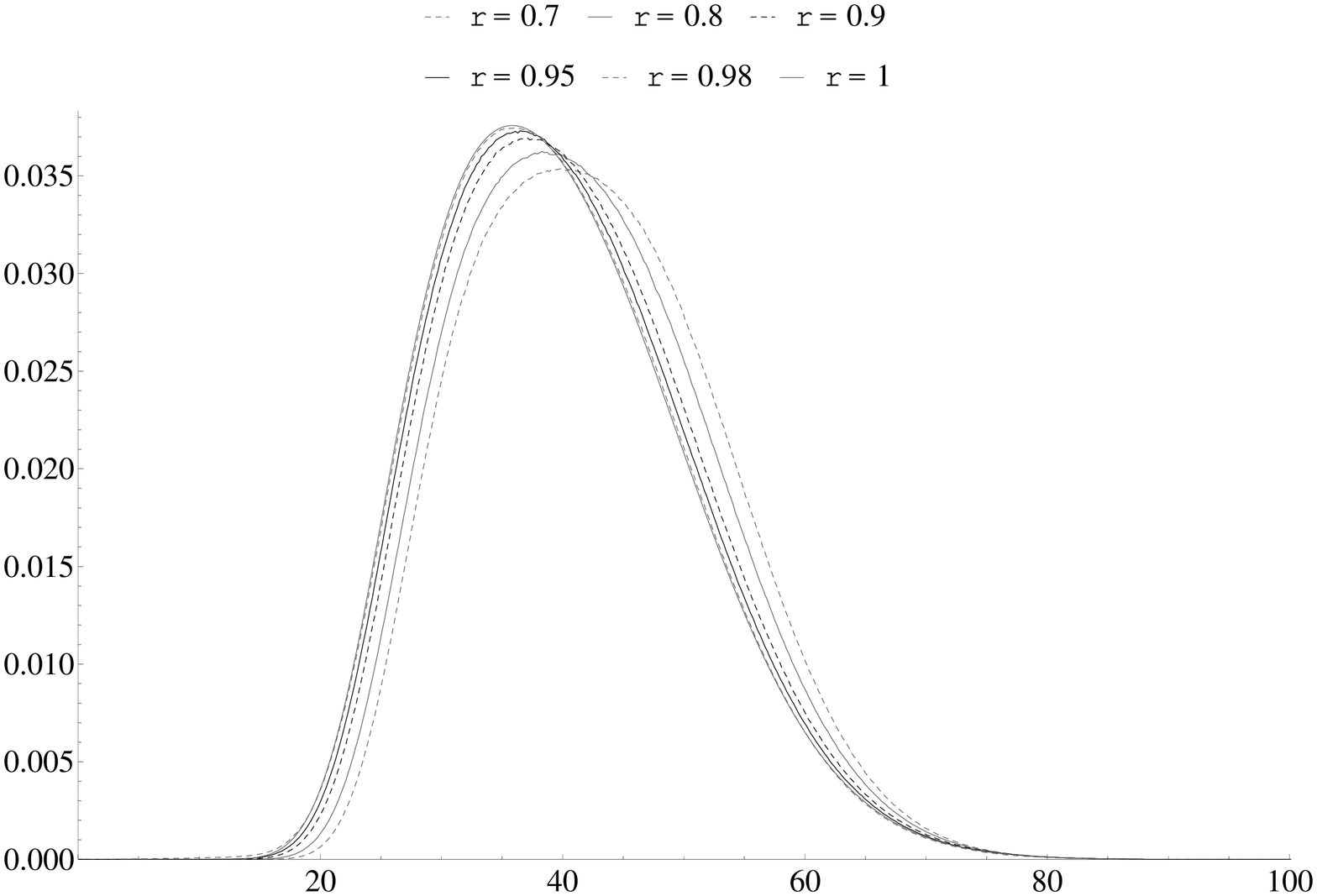}\\
$1\rightarrow 3\rightarrow 5\rightarrow 7$\\

\bigskip
\includegraphics[width=0.5\textwidth]{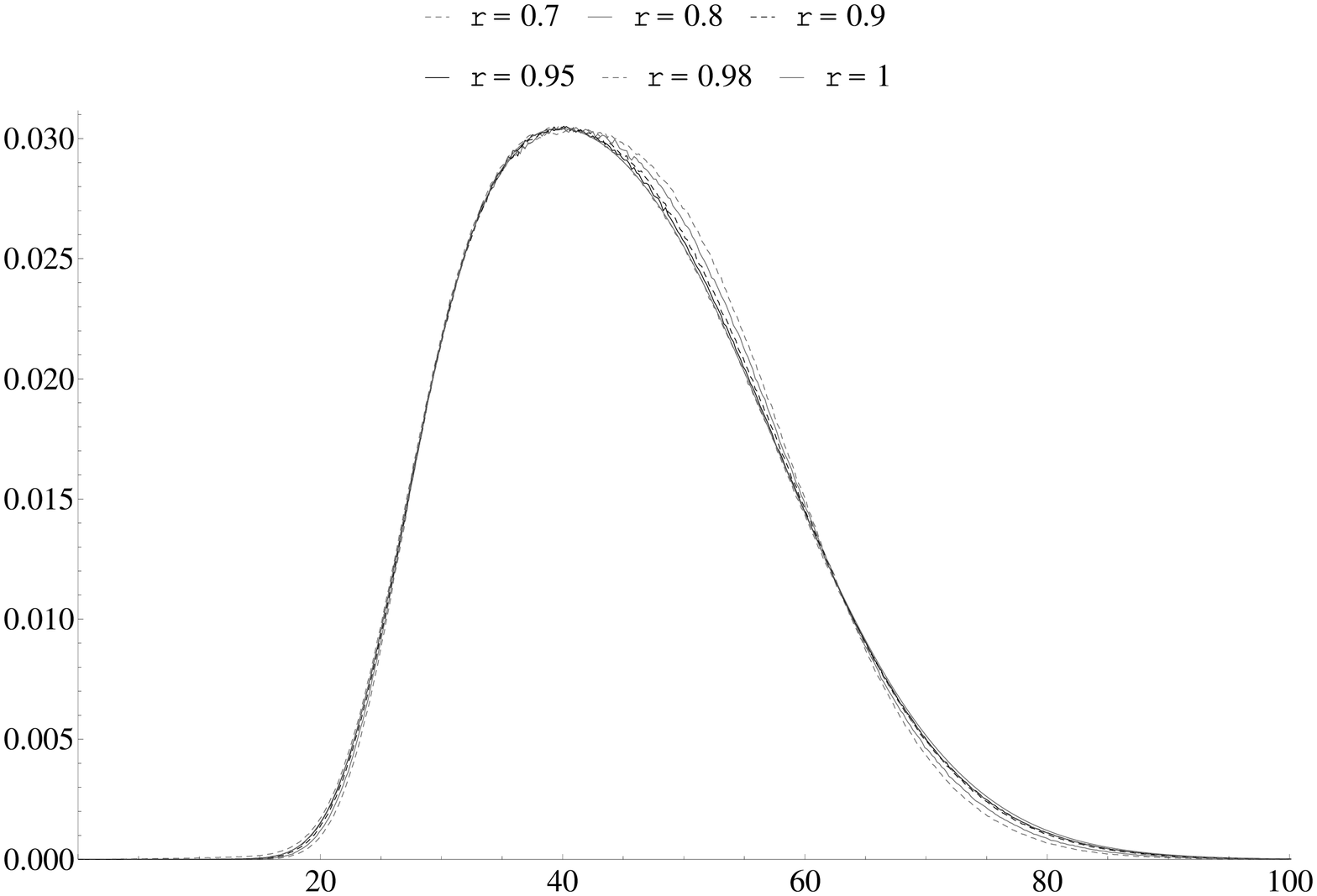}\\
 $2\rightarrow 4\rightarrow 6\rightarrow 7$\\
\end{center}
\caption{Numerical results for Example 4. The two figures show the densities of the scaled path-time distributions of the paths $1\rightarrow 3\rightarrow 5\rightarrow 7$ and  $2\rightarrow 4\rightarrow 6\rightarrow 7$.}
\label{resultsExample4figs}
\end{figure}

\begin{figure}[h]
\begin{center}
\includegraphics[width=0.5\textwidth]{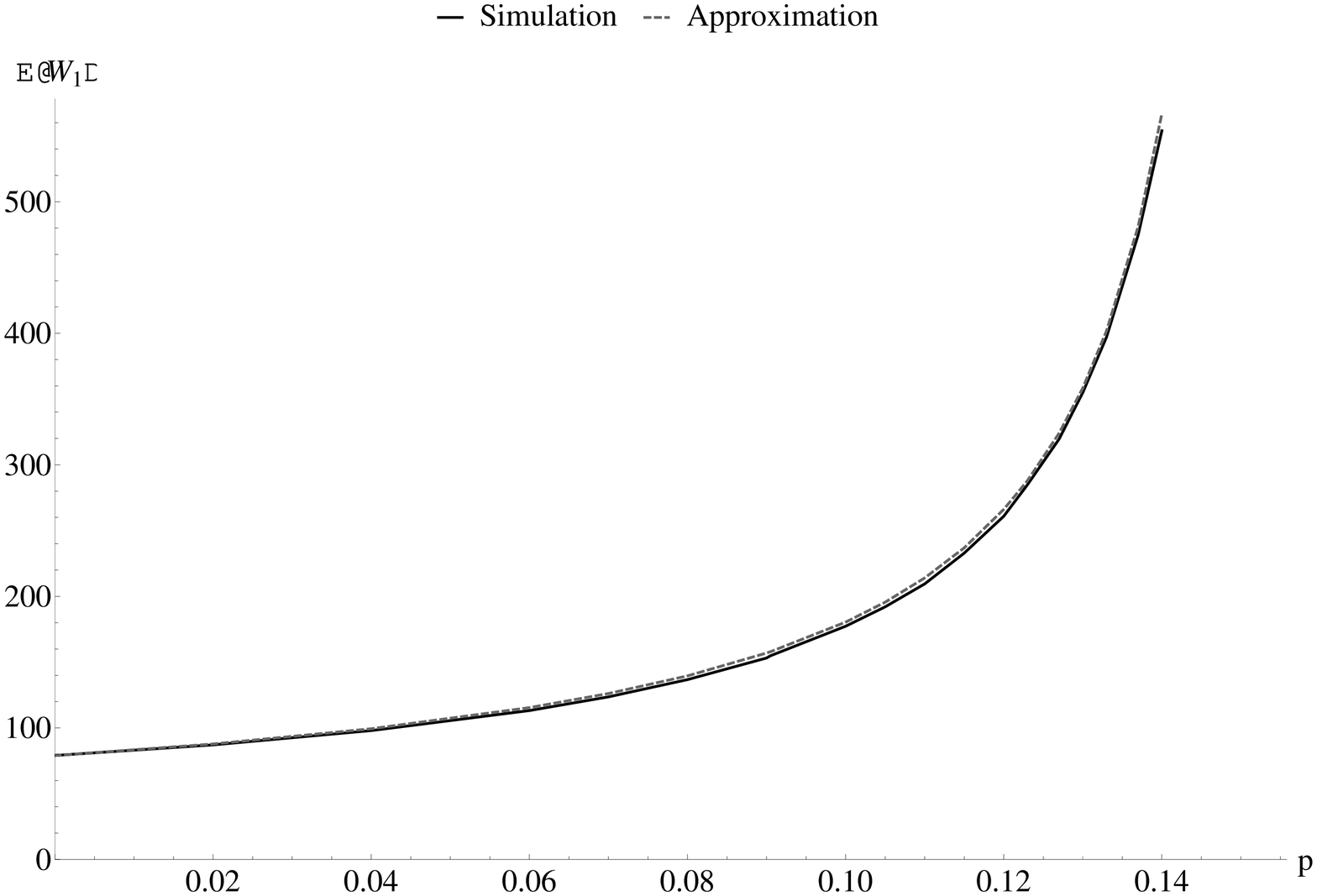}
\end{center}
\caption{The mean waiting time in the first queue as a function of the rework probability $p$, for the model in Numerical Example 4.}
\label{example4fig2}
\end{figure}

We can therefore conclude that, although the four analyzed cases come from completely different fields of application, have dissimilar parameter settings and differ significantly in system's behaviour, the derived HT asymptotics are shown to accurately describe the performance in all of these cases. Moreover, in all of these models the key performance metric is obviously the total time spent in the system by an
arbitrary customer, i.e., the path time. We hope that this observation illustrates the general nature of our framework, which extends and unifies the HT analysis of many queueing models.

\section{Conclusions and suggestions for further research}\label{sect:conclusions}

In the current paper, we have analyzed a queueing network with customer routing, where a single shared server serves the queues in a cyclic order. We have not only studied the waiting time of a customer at a certain queue, but also the path time (the total time spent in the system by an arbitrary customer traversing a specific path). The main complicating factor in this analysis is the routing of customers, which leads to non-renewal arrival processes at the queues and to strong interdependence of the waiting times at the queues. These factors prohibit an exact explicit analysis with closed-form expressions and, therefore, it is natural to resort to asymptotic estimates. That is, we have obtained easily computable expressions for both the waiting-time and path-time distribution in HT. Combining these HT asymptotics with newly developed LT limits leads to highly accurate approximations for the mean performance measures for the whole range of load values. The strength of this refined heavy-traffic approximation lies in its simplicity, which opens up interesting possibilities for optimization of the system performance with respect to the routes of the customers.

\section*{Acknowledgements}

The authors are grateful to the anonymous referees, who have provided valuable comments resulting in improved readability of the current manuscript.

\bibliographystyle{abbrvnat}

\appendix
\section*{Appendix}
\section{Proof for Poisson arrivals}
\newcommand{\beq}{\begin{equation}}
\newcommand{\eeq}{\end{equation}}

In this appendix we present a proof for Conjecture \ref{conjecturescaledqueuelengthHT} for the special case of Poisson arrivals, reformulated in the theorem below. We conclude the appendix with a short discussion.

\begin{theorem}\label{theoremscaledqueuelengthHTpoisson}
Assume that external customers arrive at $Q_i$ according to a Poisson process with rate $\lambda_i$. As $\rho \uparrow 1$, the scaled queue length $(1-\rho)L_{i}$ converges in distribution to the product of two independent random variables. The first has the same distribution as $\mathcal{L}^\textit{fluid}_{i}$, and the second random variable $\Gamma$ has the same distribution as the limiting distribution of the scaled length-biased cycle time, $(1-\rho)\bm{C_i}$. For $i=1,\dots,N; k=i,\dots,i+N-1$, and $\rho\uparrow 1$,
\begin{equation}
(1-\rho)L_{i} \limdist \Gamma \times \mathcal{L}^\textit{fluid}_{i,k}\qquad \textrm{ w.p. }\hat\rho_{k},
\label{LiHT}\\
\end{equation}
where $\Gamma$ is a random variable having a Gamma distribution with parameters $\alpha+1$ and $\delta\mu$,
and $\mathcal{L}^\textit{fluid}_{i,k}$ is the ``standardised'' number of a type-$i$ particles in the fluid model during $V_k$, introduced below \eqref{LiFluid}, independent of $\Gamma$.
The parameters of the Gamma distribution, in the case of Poisson arrivals, are defined as
\[
\alpha = r\delta/\bres, \mu=\bres,
\]
with $\delta$ as defined in Definition \ref{deltalemma}.
\end{theorem}

The proof of Theorem \ref{theoremscaledqueuelengthHTpoisson} is structured in the form of 9 small steps. Steps 1 -- 7 involve finding the limiting scaled joint queue-length distribution at visit beginnings and completions. We rely heavily on existing limiting results for polling systems and branching processes, exploiting the similarities between these models and roving server networks, closely following the framework developed in \cite{RvdM_QUESTA}. The remaining steps follow the approach in \cite{boonvdmeiwinandsRovingQuesta} and \cite{sidi2}, expressing queue lengths at arbitrary moments in terms of the queue-length distributions at visit beginnings, in heavy traffic. We start by giving some preliminary results, required for the remainder of the proof.

\paragraph{Preliminaries.}

As shown in \cite{boonvdmeiwinandsRovingQuesta,sidi2} the joint queue-length process at visit beginnings constitutes an $N$-dimensional multi-type branching process (MTBP) with immigration. As preliminaries, we give a brief overview of existing limiting results on critical MTBPs, required for the remainder of this proof. We start by introducing some additional notation.
An $N$-dimensional vector ${\ul v}$ has components $(v_1,\ldots,v_N)$, and we denote $\vert {\ul v} \vert:= \sum_{i=1}^N v_i$. Finally, $I_E$ is the indicator function on the event $E$.

Let ${\bm Z}=\{{\ul Z}_n, n=0,1,\ldots\}$ be an $N$-dimensional multi-type branching process, where ${\ul Z}_n=(Z_n^{(1)},\ldots,Z_n^{(N)})$ is an
$N$-dimensional vector denoting the state of the process in the $n$-th generation. 
The MTBP is defined by its offspring generating functions and its immigration functions. The one-step offspring generating function is denoted by
$f({\ul z})=(f^{(1)}({\ul z}),\ldots,f^{(N)}({\ul z}))$, with ${\ul z}=(z_1,\ldots,z_N)$, and 
\beq
\label{eq03}
f^{(i)}({\ul z})=\sum_{j_1,\ldots,j_N \geq 0} p^{(i)}(j_1,\ldots,j_N) z_1^{j_1}\cdots z_N^{j_N},
\eeq
where $p^{(i)}(j_1,\ldots,j_N)$ is the probability that a type-$i$ particle produces $j_k$ particles of
type $k~(i, k=1,\ldots,N)$.
The immigration function is denoted as follows: 
\beq
\label{eq04}
g({\ul z})=\sum_{j_1,\ldots,j_N \geq 0} q(j_1,\ldots,j_N) z_1^{j_1} \cdots z_N^{j_N},
\eeq
where $q(j_1,\ldots,j_N)$ is the probability that a group of immigrant consists of $j_k$ particles of type $k~(i, k=1,\ldots,N)$.
Denote the {\it mean immigration vector} by
\beq
\label{eq05}
{\ul g}:=(g_1,\ldots,g_N), ~
{\rm where~}g_i:={\partial g({\ul z}) \over \partial z_i} \vert_{{\ul z}={\ul 1}}~~(i=1,\ldots,N),
\eeq
and the {\it mean offspring matrix}, or simply {\it mean matrix},  by
\beq
\label{eq06}
{\bm{M}}=\left( m_{i,j} \right),~~~{\rm with}~m_{i,j}:={ \partial f^{(i)}({\ul z}) \over \partial z_j}
\vert_{{\ul z}={\ul 1}}~~(i,j=1,\ldots,N).
\eeq
Thus, for a given type-$i$ particle, $m_{i,j}$ is the mean number of type-$j$ ``children''
it has in the next generation. Similarly, for a type-$i$ particle, the second-order derivatives are denoted
by the matrix
\beq
\label{eq07}
{\bm K}^{(i)} = \left( k_{j,k}^{(i)} \right),~~{\rm with~}
k_{j,k}^{(i)} := { \partial^2 f^{(i)}({\ul z}) \over \partial z_j \partial z_k }
\vert_{{\ul z}={\ul 1}},~~(i,j,k=1,\ldots,N).
\eeq
Denote by
${\ul v}=(v_1,\ldots,v_N)$ and ${\ul w}=(w_1,\ldots,w_N)$ the left and right eigenvectors corresponding to the largest
real-valued, positive eigenvalue $\xi$ of ${\bm M}$, commonly referred to as the {\it maximum eigenvalue} (cf., e.g., \cite{Athreya1}), normalized such that
\beq
\label{eq08}
{\ul v}^{\top} {\ul 1} = {\ul v}^{\top} {\ul w} = 1.
\eeq
The following conditions are necessary and sufficient conditions for the ergodicity of the process ${\bm Z}$ (cf. \cite{resing93}):
$\xi<1$ and
\beq
\label{eq09}
\sum_{j_1+\cdots+j_N>0} q(j_1,\ldots,j_N) {\rm log}(j_1+\cdots+j_N) < \infty.
\eeq
Note that $\xi$ plays a role similar to $\rho$ in the roving server network. The hat-notation introduced in Section \ref{sect:model} is also used here to indicate that $x$ is {\it evaluated at $\xi=1$}. Moreover, for $\xi \geq 0$ let
\beq
\label{eq10}
\pi_0(\xi):=0,~~{\rm and}~~\pi_n(\xi):=\sum_{r=1}^n \xi^{r-2},~~n=1,2, \ldots.
\eeq

Quine \cite[Theorem 4]{quine} derives the following property for critical MTBPs.
\begin{property}
\label{htbranchingprop}
{\it
Assume that all derivatives of $f({\ul z})$ through order two exist at ${\ul z}={\ul 1}$ and that
$0 < g_i < \infty~(i=1,\ldots,N)$. Then
\beq
\label{eq13}
{1 \over \pi_n(\xi)}
\left(
\begin{array}{c}
Z_n^{(1)}\\
\vdots
\\
Z_n^{(N)}
\end{array}
\right)
\limdist
A
\left(
\begin{array}{c}
\hat{v}_1\\
\vdots
\\
\hat{v}_N
\end{array}
\right)
\Gamma(\alpha, 1)~~{\it as}~~(\xi, n)  \rightarrow (1,\infty),
\eeq
where $\hat{\ul v}=(\hat{v}_1,\ldots,\hat{v}_N)$ is the normalized the left eigenvector of $\hat{\bm M}$,
and where $\Gamma(\alpha,1)$ is a gamma-distributed random variable with scale parameter 1 and shape parameter}
\beq \label{A}
\alpha := {1 \over A} \hat{\ul g}^{\top} \hat{\ul w} = {1 \over A} \sum_{i=1}^N \hat{g}_i \hat{w}_i, ~with~~
A:= \sum_{i=1}^N \hat{v}_i
\left( \hat{\ul w}^{\top} \hat{\bm K}^{(i)} \hat{\ul w} \right)>0.
\eeq
\end{property}

\paragraph{Step 1: Characterise the embedded MTBP.}

We now return to the roving server network. In this step we show how the joint queue-length process at successive polling instants at $Q_1$ can be described as an MTBP with immigration in each state. To this end, let ${\ul X}:=\left(X_1^, \ldots, X_N \right)$
be the $N$-dimensional vector that describes the joint queue length at an arbitrary polling instant at $Q_1$.
Let $X_{i,n}$ be the number of type-$i$ customers in the system at the $n$-th polling instant of the server
at $Q_1$, for $i=1,\ldots,N$ and $n=0,1,\ldots$, and let
\beq
{\ul X}_n:=\left(X_{1,n},\ldots, X_{N,n}\right)
\eeq
be the joint queue-length vector at the $n$-th polling instant at $Q_1$.
In this appendix we will provide the results for \emph{gated service} only. The results for exhaustive service can be obtained along the exact same lines, but require significantly more complex notations (see, for example, \cite{boonvdmeiwinandsRovingQuesta}) and are omitted here.

The following result describes the process
$\{ {\ul X}_n, n=0,1,\ldots\}$ as an MTBP.\\
\ \\
\begin{theorem}\label{thm:offspringimmigration}
The discrete-time process $\{{\ul X}_n, n=0,1,\ldots\}$ constitutes a
$N$-dimensional MTBP with immigration in each state. Denote by $B_i^*(s)$ and $R_i^*(s)$ the Laplace-Stieltjes transforms (LSTs) of respectively the service times $B_i$ and the switch-over times $R_i$. The probability generating function (PGF) of the offspring function
is given by the following expression: For $\vert z_i \vert \leq 1~(\ii)$,
\beq \label{equ31}
f(\ul z):=\left(f^{(1)}(\ul z), \ldots, f^{(N)}(\ul z) \right),
\eeq
where for $\ii$ the function $f^{(i)}(\ul z)$ is the unique solution of
\beq
\label{equ32}
f^{(i)}(\ul z)
:=
B_i^* \big( \sum_{j=1}^N \lambda_j\left(1-z_j\right)\big)
\cdot
\sum_{j=1}^N p_{i,j} f^{(j)}(\ul z),
\eeq
and where the PGF of the immigration function is given by
\beq
\label{equ33}
g(\ul z)
:=
\prod_{i=1}^N R_i^* \Big( \sum_{j=1}^i \lambda_j\left(1-z_j\right) + \sum_{j=i+1}^N \lambda_j(1-f^{(j)}(\ul z))\Big).
\eeq
\end{theorem}
\begin{proof}
Relations (\ref{equ31})-(\ref{equ33}) can be obtained along the lines of Resing \cite{resing93} for the
case of classical gated service, using simple generating-function manipulations. The only difference is that the offspring function $f^{(i)}({\ul z})$ has changed in the sense that each customer served at $Q_i$ is effectively replaced {\it not only}
by all customers that arrive at the different queues during its services time (leading to PGF $B_i^*(\sum_{j=1}^N \lambda_j(1-z_j))$ in Equation \eqref{equ32}),
but {\it in addition} by a fresh customer at $Q_j$ (which creates an additional offspring $f^{(j)}({\ul z})$), with probability $p_{i,j}$.
Next, Equation (\ref{equ33}) stems from the fact that the immigration consists of the contributions of newly arriving customers
that arrive during the switch-over times $R_i, \ii$.
\end{proof}

\paragraph{Step 2: Compute the mean offspring matrix for the roving server network.}

The following result gives a characterization of the mean offspring matrix ${\bm M}$ defined in (\ref{eq06}).

\begin{lemma}\label{meanoffspringlemma}
\it The mean offspring matrix ${\bm M}=\left( m_{i,j} \right)$ can be expressed by
\beq
\label{eq33}
{\bm M}={\bm M}_1 \cdots {\bm M}_{N},
\eeq
where for $k=1,\ldots,N$, the elements of the matrix ${\bm M}_k=\left( m_{i,j}^{(k)} \right)$ are
given by: For $i,j=1,\ldots,N$, $i \neq k$,
\beq
\label{eq35}
m_{i,j}^{(k)} = I_{\{i=j\}},
\eeq
and for $i=k$,
\beq
\label{eq36}
m_{i,j}^{(k)} = \lambda_{j} b_i + p_{i,j}~{\rm for~}j=1,\ldots,N.
\eeq
\end{lemma}
\begin{proof}
The result can be obtained directly from Theorem \ref{thm:offspringimmigration} by taking the partial derivatives of the offspring function
defined in (\ref{equ31}) and (\ref{equ32}).
\end{proof}

\paragraph{Step 3: Determine the left and right eigenvectors of ${\bm M}$ at \boldmath$\rho=1$.}

The following result gives the left and right eigenvectors  (normalized according to (\ref{eq08})) of the mean offspring
matrix ${\bm M}$, defined in (\ref{eq33})-(\ref{eq36}), evaluated at $\rho=1$.

\begin{lemma}\label{eigenvectorslemma}
The right eigenvector $\hat{\ul w}$ of the mean matrix $\hat{\bm M}$, normalized such that
$\hat{\ul w}^{\top} \hat{\ul w}=1$, corresponding with maximum eigenvalue 1, is given by
\beq
\label{eq43}
\hat{\ul w} = \left(
\begin{array}{c}
\hat{w}_1 \\ \vdots \\ \hat{w}_{N}
\end{array}
\right)
:=
\vert {\ul {\tilde{b}}} \vert^{-1} {\ul {\tilde{b}}},
~ with~
{\ul {\tilde{b}}}:=\left(
\begin{array}{c}
\tilde{b}_1 \\ \vdots \\ \tilde{b}_N
\end{array}
\right).
\eeq
The corresponding left eigenvector $\hat{\ul v}$, normalized such that $\hat{\ul v}^{\top} \hat{\ul w}=1$, corresponding with maximum eigenvalue 1,  is given by
\beq
\label{eq44}
\hat{\ul v}= \left(
\begin{array}{c}
\hat{v}_1 \\ \vdots \\ \hat{v}_{N}
\end{array} \right)
:=
{\vert {\ul {\tilde{b}}} \vert \over \delta} {\hat{\ul u}},~~
with~~ {\ul u} :=\left(
\begin{array}{c}
{u}_1 \\ \vdots \\ {u}_{N}
\end{array}
\right),
~~~where~
{u}_i:={\lambda}_i \sum_{j=i}^N {\rho}_j + \sum_{j=i}^N \gamma_j p_{j,i},
\eeq
and where
\beq
\label{eq45}
\delta :=  \hat{\ul u}^{\top} \ul{\tilde{b}}
=
\sum_{i=1}^N \sum_{j=i+1}^N \hat{\rho}_i \hat{\rho}_j + \sum_{i=1}^N \tilde{b}_i \sum_{j=i}^N \hat{\gamma}_j p_{j,i}.
\eeq
\end{lemma}
\begin{proof}
First, it is readily seen by using equations (\ref{eq35})-(\ref{eq36}) and (\ref{eq02}) that for $k=1,\ldots,N$, we have
\beq
\label{eq46}
\sum_{j=1}^{N} \hat{m}_{k,j}^{(k)} \tilde{b}_j
=
\sum_{j=1}^{N}  \left (\hat{\lambda}_{j} b_k + p_{k,j} \right) \tilde{b}_j
=
b_k \sum_{j=1}^N \hat{\lambda}_j \tilde{b}_j + \sum_{j=1}^N p_{k,j} \tilde{b}_j
=
b_k
+
\sum_{j=1}^N p_{k,j} \tilde{b}_j
=
\tilde{b}_k.
\eeq
This immediately implies that $\hat{\bm M}_k \hat{\ul w} = \hat{\ul w}$ for $k=1,\ldots,N$, and hence from (\ref{eq33})
that $\hat{\bm M} \hat{\ul w}=\hat{\bm M}_1 \cdots  \hat{\bm M}_N  \hat{\ul w} = \hat{\ul w}$, which shows that $\hat{\ul w}$
indeed is a right eigenvector of $\hat{\bm M}$ with eigenvalue 1. Similar arguments can be used to show that
$\hat{\bm M}^{\top} \hat{\ul v} = \hat{\ul v}$.
\end{proof}

\begin{remark}
Although $\delta$ defined in \eqref{eq45} may appear different, at first sight, than $\delta$ defined in Definition \ref{deltalemma}, it can be shown that they are in fact identical.
\end{remark}

\paragraph{Step 4: Mean immigration function at \boldmath $\rho=1$.}
We now proceed to specify the mean immigration vector ${\ul g}$, defined in (\ref{eq08}), for the model under consideration.
Considering the evolution of the $N$-dimensional state vector as a discrete-time Markov chain
$\{{\ul X}_n, n=0,1,\ldots\}$
at successive polling instants at $Q_1$, the ``immigrants'' in the $n$-th generation are the customers present a time $n$ that
are not children of any of the customers present at time $n-1$. Denote the mean immigration vector by
${\ul g}=(g_1,\ldots,g_N)$, where $g_i$ stands for the mean number of type-$i$ immigrants.\\

\begin{lemma}\label{meanimmigrationlemma}
The mean immigration function is given by ${\ul g}=(g_1,\ldots,g_N)$, where for $j=1,\ldots,N$,
\beq
\label{eq48}
g_j = \sum_{i=1}^N r_i \left(\lambda_j I_{\{ j \leq i\}} + \sum_{k=i+1}^N \lambda_k m_{k,j} \right).
\eeq
{\it Moreover,}
\beq
\label{eq50}
\hat{\ul g}^{\top} \hat{\ul w} = \vert {\ul b} \vert^{-1}r.
\eeq
\end{lemma}
\begin{proof}
Equation (\ref{eq48}) follows directly from Theorem \ref{thm:offspringimmigration} by differentiating once with respect to $s_j$ and subsituting
${\ul s}=(1,\ldots,1)$. Next, to prove (\ref{eq50}), assume $\rho=1$. We first observe that it follows from (\ref{eq48}) that the mean number of type-$j$ customers that immigrate during a cycle is given by
\beq
\label{eq69}
\hat{g}_j = \sum_{i=1}^N r_i \left( \hat{\lambda}_j I_{\{ j \leq i\}} +
\sum_{k=i+1}^N \hat{\lambda}_k \hat{m}_{k,j} \right).
\eeq
This implies
\beq
\label{eq70}
\hat{\ul g}^{\top}\hat{\ul w}:=\vert {\ul {\tilde{b}}} \vert^{-1} \sum_{j=1}^N \hat{g}_j \tilde{b}_j
= \vert \tilde{\ul b} \vert^{-1}
\sum_{i=j}^N r_i
\left(
\sum_{j=1}^i \tilde{b}_j \hat{\lambda}_j
+
\sum_{k=i+1}^N \hat{\lambda}_k \sum_{j=1}^N \hat{m}_{k,j} \tilde{b}_j \right)
\eeq
\beq
=
\vert \tilde{\ul b} \vert^{-1}
\sum_{i=1}^N r_i
\left(
\sum_{j=1}^i \tilde{b}_j \hat{\lambda}_j
+
\sum_{k=i+1}^N \hat{\lambda}_k \tilde{b}_k \right)
=
\vert \tilde{\ul b} \vert^{-1} r \sum_{i=1}^N \hat{\rho}_i
=
\vert \tilde{\ul b} \vert^{-1} r,
\eeq
by using the definition in (\ref{eq48}) and the results in Lemma 2.
\end{proof}

\paragraph{Step 5: Determine the parameter \boldmath $A$.}
The following result gives an expression for the scaling parameter $A$, defined in (\ref{A}).
\begin{lemma}\label{scalingparameterlemma}
\beq
\label{eq51}
A = \vert \tilde{\ul b} \vert^{-1} \delta^{-1} \cdot {\tilde{b}^{\textit{res}}},
\eeq
\end{lemma}
where $\tilde{b}^\textit{res}=\frac{\E[\Btilde^2]}{2\E[\Btilde]}$ denotes the expected residual extended service time $\Btilde$ of an \emph{arbitrary} customer in the system. More precisely,
\[
\E[\Btilde^k]=\sum_{i=1}^N\lambda_i\E[\Btilde_i^k]/\sum_{j=1}^N\lambda_j.
\]
\begin{proof}
This result can be proven following the same lines as in \cite{RvdM_QUESTA} (Remark 5) for the
case of branching-type service policies for the case without customer routing (i.e., $p_{i,j}=0$ for all $i, j$).
\end{proof}

\paragraph{Step 6: Asymptotic properties for the maximum eigenvalue of \boldmath$M$.}
The following result describes the limiting behaviour of the maximum eigenvalue $\xi(\rho)$ of the matrix ${\bm M}$
defined in Lemma 1, considered as a function of $\rho$, as $\rho$ goes to 1.
\begin{lemma}
The maximum eigenvalue $\xi=\xi(\rho)$ satisfies the  following properties:\\
\ \\
{\rm (1)} $\xi<1$ if and only if $\rho<1$, $\xi=1$ if and only if $\rho=1$ and $\xi>1$ if and only if $\rho>1$;\\
{\rm (2)} $\xi(\rho)$ is a continuous function of $\rho$;\\
{\rm (3)} $\lim_{\rho \uparrow 1} \xi(\rho) = \xi(1) = 1$;\\
{\rm (4)} the derivative of $\xi(\rho)$ at $\rho=1$ is given by
\beq
\label{eq53}
\xi^{\prime}(1)=\lim_{\rho \uparrow 1}~{1-\xi(\rho) \over 1-\rho}={1 \over \delta},
\eeq
where $\delta$ is defined in (\ref{eq45}).
\end{lemma}
\begin{proof}
The proof can be obtained by following the approach in \cite{RvdM_QUESTA} (Section 3.2), which proves for the
case without customer routing (i.e., $p_{i,j}=0$ for all $i, j$).
\end{proof}

\paragraph{Step 7: The joint queue-length vector at visit beginnings and completions.}
We are now ready to present the HT result for the state vector
at polling instants. Without loss of generality, we focus on the evolution of
the state vector at embedded polling instants at $Q_1$.

\begin{theorem}
The joint queue-length vector at polling instants at $Q_1$ has the following asymptotic behaviour:
\beq
\label{eq72}
(1-\rho)
\left(
\begin{array}{c}
X_1 \\ \vdots \\ X_N
\end{array}
\right)
\limdist
{\tilde{b}^{\textit{res}}}
~
{1 \over \delta}
\left(
\begin{array}{c}
\hat{u}_1 \\ \vdots \\ \hat{u}_N
\end{array}
\right)
~
\Gamma(\alpha, 1)~~~(\rho \uparrow 1),
\eeq
where
\beq
\label{eq73}
\alpha = r \delta /\tilde{b}^{\textit{res}},
\eeq
and where $\hat{u}_i~(\ii)$ and $\delta$ are defined in (\ref{eq44}) and (\ref{eq45}).
\end{theorem}
\begin{proof}
First, note that the process that describes the evolution of the state vector
$\{{\ul X}_n, n=0,1,\ldots\}$ at successive polling instants
at $Q_1$ constitutes an $N$-dimensional MTBP with
offspring function $f({\ul s})$ and immigration function $g({\ul z})$ defined in Theorem \ref{thm:offspringimmigration}, and with mean matrix ${\bm M}$
defined in Lemma \ref{meanoffspringlemma}. Moreover, from Theorem \ref{thm:offspringimmigration} it is readily verified that the assumptions of Property 1 on the finiteness
of the second-order derivatives of $f({\ul z})$ and the mean immigration function ${\ul g}$ are satisfied.
Then using Property \ref{htbranchingprop} it follows that
\beq
\label{eq58}
{1 \over \pi_n(\xi)}\cdot
{\ul X}_n^{\top}
\limdist
A \cdot
\hat{\ul v} \cdot \Gamma(\alpha, 1)~~{\it as}~~(\xi, n) \rightarrow (1, \infty),
\eeq
where $A$, $\hat{\ul v}$ and $\alpha$ are given in (\ref{A}).
Hence, translating this to the polling model and using Lemmas 1-5, it readily follows from (\ref{eq58}) that
\beq
\label{eq59}
(1-\rho) {\ul X}_n^{\top}
\limdist \delta \cdot A \cdot
\hat{\ul v} \cdot \Gamma(\alpha, 1)~~{\it as}~~(\rho, n) \rightarrow (1, \infty),
\eeq
where expressions for $\delta$, $A$, $\hat{\ul v}$ and $\alpha$ are given in (\ref{eq45}), (\ref{eq51}), (\ref{eq44}) and (\ref{eq73}). Combining these expressions leads to the result.
\end{proof}

\begin{corollary}
Let $X_{i,k}^\textit{scaled}:=\lim_{\rho\uparrow1}(1-\rho)X_{i,k}$ denote the scaled number of customers in $Q_i$ at the beginning of a visit to $Q_k$. Its LST is equal to
\begin{equation}
\label{XikScaled}
\E\left[e^{-\omega X_{i,k}^\textit{scaled}}\right] = \begin{cases}
\left( \frac{\delta/\bres}{\delta/\bres+\omega\sum_{j=i}^{k-1}\hat\rho_j\hat\gamma_{i,j}}\right)^{r\delta/\bres}
\qquad&k=i+1,\dots,i+N-1,\\[2ex]
\left( \frac{\delta/\bres}{\delta/\bres+\omega\hat\gamma_{i}}\right)^{r\delta/\bres}
\qquad&k=i.
\end{cases}
\end{equation}
\end{corollary}

\paragraph{Step 8: The marginal queue-length distribution.} Let $L_{i,k}^\textit{scaled}:=\lim_{\rho\uparrow1}(1-\rho)L_{i,k}$ denote the scaled number of customers in $Q_i$ at an arbitrary moment during a visit to $Q_k$, and let $L_{i}^\textit{scaled}:=\lim_{\rho\uparrow1}(1-\rho)L_{i}$ denote the scaled number of customers in $Q_i$ at an arbitrary moment.
\begin{theorem}
\label{scaledqueuelengththeorem}
\begin{equation}
\E\left[e^{-\omega L_{i}^\textit{scaled}}\right] = \sum_{k=1}^N \hat\rho_k \E\left[e^{-\omega L_{i,k}^\textit{scaled}}\right],
\label{scaledLi}
\end{equation}
where
\begin{align}
\E\left[e^{-\omega L_{i,i}^\textit{scaled}}\right] &=
\frac{\left( \frac{\delta/\bres}{\delta/\bres+\omega\hat\rho_i\hat\gamma_{i,i}}\right)^{r\delta/\bres}-
\left( \frac{\delta/\bres}{\delta/\bres+\omega\hat\gamma_{i}}\right)^{r\delta/\bres}}
{(\hat\gamma_i-\hat\gamma_{i,i}\hat\rho_i)r\omega},
\label{scaledLii}\\
\E\left[e^{-\omega L_{i,k}^\textit{scaled}}\right] &=
\frac{\left( \frac{\delta/\bres}{\delta/\bres+\omega\sum_{j=i}^{k-1}\hat\rho_j\hat\gamma_{i,j}}\right)^{r\delta/\bres}-
\left( \frac{\delta/\bres}{\delta/\bres+\omega\sum_{j=i}^{k}\hat\rho_j\hat\gamma_{i,j}}\right)^{r\delta/\bres}}
{\hat\rho_k\hat\gamma_{i,k}r\omega},\qquad i\neq k.
\label{scaledLik}
\end{align}
\end{theorem}
\begin{proof}
To prove Theorem \ref{scaledqueuelengththeorem} we need the following results, summarised in \eqref{Li}--\eqref{Yik}, which directly follow from Equations (3.6)--(3.10) in \cite{sidi2}.

\begin{equation}
\E[z^{L_i}]= \sum_{k=1}^N\left(\hat\rho_k\E[z^{L_i^{(V_k)}}]+(1-\rho)\frac{r_j}{r}\E[z^{L_i^{(R_k)}}]\right),\label{Li}
\end{equation}
where $L_i^{(V_k)}$ and $L_i^{(R_k)}$ denote the number of customers in $Q_i$ at an arbitrary moment during $V_k$ and $R_k$ respectively.

Denote by $X_{i,k}^*(z)$ and $Y_{i,k}^*(z)$ the PGF of the number of customers in $Q_i$ at the beginning and end of $V_k$, respectively.
\begin{align}
L_i^{(V_i)}(z) &= \big(X_{i,i}^*(z) - Y_{i,i}^*(z)\big)\frac{(1-\rho)(1-B^*_i\big(\lambda_i(1-z)\big))}{\lambda_i(1-z)\rho_i r\big(1 - B^*_i(\lambda_i(1-z))(1-p_{i,i}+p_{i,i}z)/z)} \label{LiVi}\\
L_i^{(V_k)}(z) &= \big(X_{i,k}^*(z) - Y_{i,k}^*(z)\big)\frac{(1-\rho)(1-B^*_k\big(\lambda_i(1-z)\big))}{\lambda_i(1-z)\rho_k r\big(1 - B^*_k(\lambda_i(1-z))(1-p_{k,i}+p_{k,i}z))} ,\qquad \text{ for }k\neq i.\label{LiVk}
\end{align}
Note that
\begin{equation}
Y_{i,k}^*(z) = \frac{X_{i,k+1}^*(z)}{R_k^*(\lambda_i(1-z))}.
\label{Yik}
\end{equation}
Our goal is to find the limiting distribution of $(1-\rho)L_i$ as $\rho\uparrow 1$. In the limit, when substituting $z=e^{ -\omega(1-\rho)}$ in \eqref{Li} and letting $\rho\uparrow 1$, the terms that correspond to the queue lengths during switch-over times vanish, caused by the $(1-\rho)$ term. Intuitively this is exactly what one would expect, as switch-over times become negligible in heavy traffic. In order to prove \eqref{scaledLii} and \eqref{scaledLik}, we substitute $z=e^{ -\omega(1-\rho)}$, $\lambda_i=\hat\lambda_i\rho$, and $\rho_i=\hat\rho_i\rho$ in \eqref{LiVi} and  \eqref{LiVk} respectively, and evaluate the Taylor series of the resulting functions near $\rho=1$.

For the case $k=i$ this results in
\begin{equation}
\lim_{\rho\uparrow1}\E[e^{-\omega(1-\rho)L_i^{(V_i)}}]= \frac{\E\left[e^{-\omega X_{i,i+1}^\textit{scaled}}\right] - \E\left[e^{-\omega X_{i,i}^\textit{scaled}}\right]}{\omega r \hat\rho_i(1-\hat\lambda_ib_i-p_{i,i})/b_i},
\end{equation}
and for the case $k\neq i$ we obtain
\begin{equation}
\lim_{\rho\uparrow1}\E[e^{-\omega(1-\rho)L_i^{(V_k)}}]= \frac{\E\left[e^{-\omega X_{i,k}^\textit{scaled}}\right] - \E\left[e^{-\omega X_{i,k+1}^\textit{scaled}}\right]}{\omega r \hat\rho_k(\hat\lambda_ib_i+p_{k,i})/b_k}.
\end{equation}
After substitution of \eqref{XikScaled} this leads to the following result:
\begin{align}
\lim_{\rho\uparrow1}\E[e^{-\omega(1-\rho)L_i^{(V_i)}}]&=
\frac{\left( \frac{\delta/\bres}{\delta/\bres+\hat\rho_i\hat\gamma_{i,i}\omega}\right)^{r\delta/\bres} - \left( \frac{\delta/\bres}{\delta/\bres+\hat\gamma_{i}\omega}\right)^{r\delta/\bres}}{ r (\hat\gamma_i-\hat\rho_i\hat\gamma_{i,i})\omega},
\label{LiViscaled1}\\
\lim_{\rho\uparrow1}\E[e^{-\omega(1-\rho)L_i^{(V_k)}}]&=
\frac{\left( \frac{\delta/\bres}{\delta/\bres+\omega\sum_{j=i-N}^{k-1}\hat\rho_j\hat\gamma_{i,j}}\right)^{r\delta/\bres} - \left( \frac{\delta/\bres}{\delta/\bres+\omega\sum_{j=i-N}^{k}\hat\rho_j\hat\gamma_{i,j}}\right)^{r\delta/\bres}}{ r \hat\rho_k\hat\gamma_{i,k}\omega}.
\label{LiViscaled2}
\end{align}

\paragraph{Step 9: Proving Theorem~\ref{theoremscaledqueuelengthHTpoisson}.}
To prove that Theorem~\ref{theoremscaledqueuelengthHTpoisson} agrees with the results obtained in Step 8, we need the following, well-known property that is frequently used in HT limits of polling systems.
\begin{property}
Let $\Gamma$ be a random variable with a Gamma distribution with shape parameter $\alpha^*$ (to avoid confusion since we already have defined a parameter $\alpha$) and rate parameter $\beta$. Let $U$ be a uniform random variable on the interval $[a, b]$, independent of $\Gamma$. The LST of the product $U\times I$ is equal to
\begin{equation}
\E[e^{-\omega U I}] =
\frac{\left(\frac{\beta/a }{\beta/a+\omega }\right)^{\alpha^* -1}-\left(\frac{\beta/b }{\beta /b+\omega }\right)^{\alpha^* -1}}{(b-a)(\alpha^* -1)   \omega/\beta }.
\label{gammauniform}
\end{equation}
\end{property}
Theorem~\ref{theoremscaledqueuelengthHTpoisson} states that the scaled queue length $L_i^\textit{scaled}$ is distributed as $L_{i,k}^\textit{scaled}$ with probability $\hat\rho_k$. This is in accordance with what one would obtain when substituting $z=e^{ -\omega(1-\rho)}$ in \eqref{Li} and letting $\rho\uparrow 1$.

The distribution of $L_{i,k}^\textit{scaled}$ is the product of two random variables. The first is a Gamma($\alpha+1$, $\delta\mu$) distribution,
where $\alpha+1$ and $\delta\mu$ are the parameters of the Gamma distribution of the scaled, length-biased cycle time, discussed in Theorem~\ref{theoremscaledqueuelengthHTpoisson}. When the external arrival process is a Poisson process, the parameter $\sigma^2$ is equal to $\btilde^{(2)}/\btilde$, which means that
\begin{equation}
\alpha = r\delta/\bres, \mu=\bres,
\label{alphamupoisson}
\end{equation}
with $\delta$ as defined in Definition \ref{deltalemma}.

The second random variable in the product is $\mathcal{L}^\textit{fluid}_{i,k}$, which is uniformly distributed. Remember that $\mathcal{L}^\textit{fluid}_{i,k} = L^\textit{fluid}_{i,k}/c$, so taking $c=1$ in \eqref{LikDist} gives the parameters of the uniform distribution for each of the $\mathcal{L}^\textit{fluid}_{i,k}$. Substituting the following values in \eqref{gammauniform} leads to the LST of the limiting distribution of the scaled number of customers in $Q_i$ at an arbitrary epoch during $V_k$,
\begin{align*}
\alpha^*&=r\delta/\bres+1, &&\beta=\delta\bres, &&\\
a &= \hat\rho_j \hat\gamma_{i,j}, &&b = \hat\gamma_{i}, &&\qquad (i=k)\\
a &= \sum_{j=i}^{k-1} \hat\rho_j \hat\gamma_{i,j}, &&b = \sum_{j=i}^{k} \hat\rho_j \hat\gamma_{i,j}, &&\qquad (i\neq k)
\end{align*}
It is quickly verified that these substitutions result in an expression completely equivalent to \eqref{LiViscaled1} for the case $i=k$, and to \eqref{LiViscaled2} for $i\neq k$. This concludes the proof of Theorem~\ref{theoremscaledqueuelengthHTpoisson}.
\end{proof}

\subsection*{Discussion}

In this appendix we have provided a proof for the scaled queue-length distributions in heavy traffic, relying on the framework developed in \cite{RvdM_QUESTA}. In his paper, Van der Mei uses the distributional form
of Little's Law to obtain the distributions of the scaled waiting times in polling models with Poisson arrivals. We note that the distributional form of Little's Law cannot be applied to our model because of the internal routing, as discussed extensively in~\cite{boonvdmeiwinandsRovingQuesta}. In \cite{boonvdmeiwinandsRovingQuesta} nevertheless a mathematical framework has been developed to derive the LSTs of the steady-state waiting-time distributions. As such, this framework provides an excellent basis to prove the HT limits for waiting times without resorting to the distributional form of Little's Law. The derived LSTs for steady-state waiting-time distributions are given in the form of recursive expressions, which will simplify to the elegant closed-form expressions presented in this paper, after taking the HT limit.

For \emph{path times}, there are currently no steady-state results available at all, due to the complex dependencies between successive visit times. In order to prove the HT limit of the scaled path times, one would first have to apply the techniques in~\cite{boonvdmeiwinandsRovingQuesta} to find path-time LSTs in steady state, which is a separate study by itself.

We conclude this discussion by noting that the \emph{mean} waiting times can easily be obtained by applying Little's Law, which does not require the assumption of Poisson arrivals.

\end{document}